\documentclass[12pt]{article}

\usepackage{amsmath,amsfonts,amssymb,amsthm,amscd}
\usepackage{hyperref} 
\usepackage[a4paper, textheight=660pt, textwidth=445pt]{geometry}
\usepackage{color,xcolor} 
\title{On non-compact $p$-adic definable groups}
\date{\today}

\author{Will Johnson and Ningyuan Yao}

\newtheorem{Thm}{Theorem}[section]
\newtheorem{Prop}[Thm]{Proposition}
\newtheorem{Lemma}[Thm]{Lemma}
\newtheorem{Cor}[Thm]{Corollary}
\newtheorem{Fact}[Thm]{Fact}

\newtheorem{Conj}[Thm]{Conjecture}

\newtheorem*{Claim1}{Claim 1}
\newtheorem*{Claim2}{Claim 2}

\newtheorem*{Claim}{Claim}

\theoremstyle{definition}
\newtheorem{Def}[Thm]{Definition}
\newtheorem{Notation}[Thm]{Notation}

\newtheorem{Rmk}[Thm]{Remark} 

\theoremstyle{remark}
\newtheorem*{Acks}{Acknowledgments}

\newenvironment{claimproof}[1][\proofname]
               {
                 \proof[#1]
                 
               }
               {
                 \endproof
               }

\newcommand{\R}{\mathbb R}
\newcommand{\B}{{\cal B}}
\newcommand{\Q}{{\mathbb Q}_p}
\newcommand{\Z}{{\mathbb Z}_p}

\newcommand{\N}{\mathbb N}

\newcommand{\M}{\mathbb M}

\newcommand{\la}{\mathcal{L}}

\newcommand{\mupn}{{\mu^N\cdot p^N}}
\newcommand{\mup }{{\mu\cdot  p }}

\newcommand{\stn }{{\st_N^{\M}}}

\newcommand{\ra}{\rightarrow}

\newcommand{\sq}{\subseteq}

\DeclareMathOperator{\dpr}{dp-rk}

\DeclareMathOperator{\Th}{Th}

\DeclareMathOperator{\id}{id}
\DeclareMathOperator{\stab}{stab}
\DeclareMathOperator{\Stab}{Stab}

\DeclareMathOperator{\tp}{tp}

\DeclareMathOperator{\dcl}{dcl}
\DeclareMathOperator{\st}{st}

\DeclareMathOperator{\acl}{acl}
\DeclareMathOperator{\alg}{alg}

\begin{document}
\maketitle

\begin{abstract}
In \cite{Y.-Peterzil-and-C.-Steinhorn}, Peterzil and Steinhorn proved
that if a group $G$ definable in an $o$-minimal structure is not
definably compact, then $G$ contains a definable torsion-free subgroup
of dimension one. We prove here a $p$-adic analogue of the
Peterzil-Steinhorn theorem, in the special case of abelian groups.

Let $G$ be an abelian group definable in a $p$-adically closed field
$M$. If $G$ is not definably compact then there is a definable
subgroup $H$ of dimension one which is not definably compact.  In a
future paper we will generalize this to non-abelian $G$.
\end{abstract}

\section{Introduction}

In \cite{Y.-Peterzil-and-C.-Steinhorn}, Peterzil and Steinhorn prove
that if $G$ is a definable group in an $o$-minimal structure $M$, and
$G$ is not definably compact, then $G$ has a definable 1-dimensional
subgroup $H$ that is not definably compact.  To prove this, they take
a continuous unbounded definable curve $I : [0,+\infty) \to G$ and
take $H$ to be the ``tangent line at $\infty$.''  This can be made
precise using the language of \emph{$\mu$-types} and
\emph{$\mu$-stabilizers} developed later by Peterzil and Starchenko
\cite{Y.-Peterzil-and-S.-Starchenko}.  Say that two complete types $q,
r \in S_G(M)$ are ``infinitesimally close'' if there are realizations
$a \models q$ and $b \models r$ such that $ab^{-1}$ is infinitesimally
close to $\id_G$ (that is, $ab^{-1}$ is contained in every
$M$-definable neighborhood of $\id_G$).  This is an equivalence
relation on $S_G(M)$, and equivalence classes are called
``$\mu$-types.''  The ``$\mu$-stabilizer'' $\stab^\mu(q)$ of $q \in
S_G(M)$ is the stabilizer of the $\mu$-type of $q$.

With these definitions, the ``tangent line of $I$ at $\infty$'' is
simply the $\mu$-stabilizer of the type on $I$ at infinity, an
unbounded 1-dimensional definable type.  (Here, we say that a type $q
\in S_G(M)$ is ``unbounded'' if no formula in $q$ defines a definably
compact subset of $G$.)  Peterzil and Steinhorn essentially show
that the $\mu$-stabilizer of an unbounded 1-dimensional definable type
is a torsion free non-compact definable subgroup of dimension 1.  More
generally, in \cite{Y.-Peterzil-and-S.-Starchenko}, Peterzil and
Starchenko consider a general definable type $q \in S_G(M)$, showing
that $\stab^\mu(q)$ is a torsion-free definable group of a certain
dimension.

It is natural to ask whether analogous results hold in the theory
$p$CF ($p$-adically closed fields).  There are many formal
similarities between $p$CF and o-minimal theories, especially RCF
(real closed fields).  In both settings, definable groups can be
regarded as real or $p$-adic Lie groups
\cite{Pillay-G-in-o,Pillay-G-in-p}, and are locally isomorphic to real
or $p$-adic algebraic groups \cite{H-P-groups-in-local-fields}. In both
the real and $p$-adic contexts, definable sets have a dimension which
has a topological description as well as an algebraic description (the
algebro-geometric dimension of the Zariski closure). On the other
hand, definable connectedness behaves very differently in the two
settings.

In this paper, we restrict our attention to one-dimensional definable
types, as in the original work of Peterzil and Steinhorn
\cite{Y.-Peterzil-and-C.-Steinhorn}.  Unfortunately, we must also
assume that $G$ is ``nearly abelian'' for most of our theorems.
\begin{Def}\label{d-na}
  Let $G$ be a definable group in a model of $p$CF. 
  $G$ is \emph{nearly abelian} if there is a definably compact
  definable normal subgroup $K \subseteq G$ with $G/K$ abelian.
\end{Def} 
See Definition~\ref{defdefc} for a precise definition of ``definable compactness,'' and Propositions~\ref{fixed-def} and \ref{sp-char} for some equivalent conditions.

Our main results are as 
follows:
\begin{Thm}\label{first-main}
  Let $G$ be a definable group over a $p$-adically closed field $M$.
  If $G$ is not definably compact and $G$ is nearly abelian, then there is a
  1-dimensional definable subgroup $H \subseteq G$ that is not
  definably compact.
\end{Thm}
We plan to
generalize Theorem~\ref{first-main} to non-abelian groups in a future
paper.

\begin{Thm}\label{second-main}
  Suppose that $G$ is a definable group over an
  $\aleph_1$-saturated $p$-adically closed field $M$.  Then for any
  definable unbounded 1-dimensional type $r \in S_G(M)$, the
  $\mu$-stabilizer $\stab^\mu(r)$ is a 1-dimensional type-definable
  subgroup of $G$.  If $G$ is abelian (or nearly abelian), then
  $\stab^\mu(r)$ is unbounded.
\end{Thm}
Here, a set or type is ``bounded'' if it is contained in a definably compact set, and ``unbounded'' otherwise (Definition~\ref{def-bounded}).
The assumption on saturation is necessary.  For example, suppose $M =
\Q$, $G$ is the multiplicative group, and $r \in S_G(\Q)$ is one of
the definable types consistent with $\{x \mid v(x) < \mathbb{Z}\}$.
Then $\stab^\mu(r)$ is the intersection of all $n$-th powers $P_n=\{x
\mid x\neq 0\wedge \exists (x=y^n)\}$, which is the trivial group
$\{1\}$.

We can also say something when $M$ is not saturated, but we will need
a few more definitions from \cite{Y.-Peterzil-and-S.-Starchenko}.  Fix
a group $G$ definable in a $p$-adically closed field $M$.  For any
partial type $\Sigma(x)$ in $G$, and any $\la$-formula $\phi(x;y)$,
let $\stab_\phi(\Sigma)$ denote
\begin{equation*}
  \bigcap_{b \in M^k} \stab \{ g \in G(M) \mid \Sigma \vdash \phi(g x; b) \}.
\end{equation*}
(This can be understood as the stabilizer of the $\phi(z \cdot x;
y)$-type generated by $\Sigma(x)$.)  It turns out that $\stab(\Sigma)
= \bigcap_{\phi \in \la} \stab_\phi(\Sigma)$.  Now suppose that $r$ is
a type in $S_G(M)$.  Let $\mu$ be the partial type of
``infinitesimals,'' that is, the set of $\la_M$-formulas definining
neighborhoods of $\id_G$.  Let $\mu \cdot r$ be the partial type such
that $(\mu \cdot r)(N) = \mu(N) \cdot r(N)$ for sufficiently saturated
$N \succ M$.  It turns out that
\begin{equation*}
  \stab^\mu(r) = \stab(\mu \cdot r) = \bigcap_{\phi \in \la}
  \stab_\phi(\mu \cdot r).
\end{equation*}
Moreover, when $r$ is definable, the groups $\stab_\phi(\mu \cdot r)$
are definable, and $\stab(\mu \cdot r)$ is type-definable.  (This is
the reason why $\stab^\mu(r)$ is type-definable in
Theorem~\ref{second-main}.  In the o-minimal case, there is a
descending chain condition on definable groups, which ensures that
$\stab^\mu(r)$ is \emph{definable} in
\cite{Y.-Peterzil-and-S.-Starchenko}.)

\begin{Thm}\label{third-main}
  Suppose that $G$ is a definable group over a $p$-adically closed
  field $M$.  Let $r \in S_G(M)$ be a definable unbounded 1-dimensional type.
  Then there is a formula $\phi \in \la$ such that $\stab_\phi(\mu
  \cdot r)$ is a 1-dimensional definable subgroup of $G$.  When $G$ is
  abelian (or nearly abelian), $\stab_\phi(\mu \cdot r)$ is unbounded.
\end{Thm}
Our proofs of these theorems are based on the original proofs of
Peterzil and Steinhorn \cite{Y.-Peterzil-and-C.-Steinhorn}, though
several important changes are necessary.  First of all, the
$\mu$-stabilizer $\stab^\mu(r)$ is no longer definable, but merely
type-definable, as mentioned above.  For this reason, it is necessary
to compute the stabilizers in an $|M|^+$-saturated elementary
extension $N \succ M$.

A more serious problem arises when trying to generalize \cite[Lemma
  3.8]{Y.-Peterzil-and-C.-Steinhorn}.  This lemma, which is used to
show that $\stab^\mu(p) \ne \{\id_G\}$, roughly says the following: if
$I$ is a curve tending to infinity and $B$ is an annulus around
$\id_G$, then $g \cdot B \cap I \ne \emptyset$ for all $g \in I$.
This follows by a simple connectedness argument ($I$ is connected, so
it must cut across the annulus $g \cdot B$ on its way from $g$ to
infinity).  This argument fails critically in the totally disconnected
$p$-adic context.  In Section~\ref{sec:gaps} we develop an alternative
argument to replace \cite[Lemma 3.8]{Y.-Peterzil-and-C.-Steinhorn}.
Unfortunately, the argument only works properly in the abelian (or
near-abelian) case.

\subsection{Notation and conventions}

We shall assume a basic knowledge of model theory, including basic
notions such as definable types, saturation, heirs, and so on.  Good
references are \cite{M-book,P-book}. We refer to the excellent survey
\cite{Luc-Belair} as well as \cite{O-P, H-P-groups-in-local-fields}
for the model theory of the $p$-adic field $(\Q, +, \times, 0, 1)$. In
fact, \cite{O-P} and \cite{H-P-groups-in-local-fields} are also good
references for the model theoretic background required for the current
paper.

Let $T$ be a theory in some language $\la$.  We write $\M$ for a
monster model of $T$, in which every type over a small subset
$A\subseteq \M$ is realized, where ``small" means $|A|< \kappa$ for
some big enough cardinal $\kappa$.  The letters $M,N, M'$ and $ N'$
will denote small elementary submodels of $\M$.  We will use $x, y, z$
to mean arbitrary $n$-tuples of variables and $a, b, c \in \M$ to
denote $n$-tuples in $\M^n$ with $n\in \N$. Every formula is an
$\la_\M$-formula. For an $\la_M$-formula $\phi(x)$, $\phi(M)$ denotes
the definable subset of $M^{|x|}$ defined by $\phi$, and a set
$X\subseteq M^n$ is definable if there is an $\la_M$-formula $\phi(x)$
such that $X=\phi(M)$. If $M\prec N\prec \M$, and $X\subseteq
N^n$ is defined by a formula $\psi$ with parameters from $M$, then
$X(M)$ and $X(\M)$ will denote $\psi(M)$ and $\psi(\M)$
respectively; these are clearly definable subsets of $M^n$ and $\M^n$
respectively.

Following \cite[Definition~2.12]{Y.-Peterzil-and-S.-Starchenko}, we
say that a partial type $\Sigma$ is \emph{$A$-definable} or
\emph{definable over $A$} if for every formula $\phi(x;y)$, there is
an $\la_A$-formula $\psi(y)$ such that
\begin{equation*}
  \Sigma(x) \vdash \phi(x;b) \iff M \models \psi(b)
\end{equation*}
for all $b \in M$.  We will denote the formula $\psi(y)$ by $(d_\Sigma
x) \phi(x,y)$, thinking of $d_\Sigma$ as a quantifier.
The map $\phi(x;y) \mapsto (d_\Sigma x) \phi(x,y)$ is called the \emph{definition schema} of $\Sigma(x)$. 

If $\Sigma(x)$ is a definable partial type over $M$, and $N\succ M$,
then $\Sigma^N$ will denote the canonical extension of $\Sigma$ by definitions,
i.e., the following partial type over $N$:
\begin{equation*}
  \Sigma^N = \{\phi(x;a) \in \la_N \mid N \models (d_\Sigma x) \phi(x,a)\}.
\end{equation*}
When $p$ is a complete definable type over $M$, the canonical extension $p^N$ is the same thing as
the unique heir of $p$ over $N$.

For a definable set $D\sq M^n$, and $\phi(x)$ an $\la_\M$-formula, we
say that $\phi(x)$ is a $D$-formula if $\M\models \phi(x) \implies
x\in D(\M)$. A partial type $q(x)$ (over a small subset) is a $D$-type
if $q(x)\vdash x\in D(\M)$. We write $S_D(M)$ for the space of
complete $D$-types over $M$.

We consider $\Q$ as a structure in the language of rings
$\la = \la_r =\{+,\times,-,0,1\}$.  The valuation ring $\Z$ is definable in
$\Q$. The valuation group $(\mathbb Z,+,<)$ and the valuation
$v:\Q\ra\mathbb Z\cup\{\infty\} $ are interpretable.  A
\emph{$p$-adically closed field} is a model of $p\mathrm{CF} := \Th(\Q)$.  For
any $M\models p\mathrm{CF}$, $R(M)$ will denote the valuation ring, and
$\Gamma_M$ will denote the value group.  By \cite{Macintyre}, $p$CF
admits quantifier elimination after adjoining predicates $P_n$ for the
$n$-th power of the multiplicative group for all $n\in \N^+$.  The
theory $p$CF also has definable Skolem functions \cite{L-Dries-Skolem}.

The $p$-adic field $\Q$ is a locally compact topological field, with  basis given by
the sets
\[\B(a,n) = \{x\in \Q \mid x \neq a \wedge v(x - a) \geq n\}\]
for $a\in \Q$ and $n\in \mathbb Z$. The valuation ring $\Z$ is
compact. The topology is definable (as in Section \ref{adc} below), so it
extends to any $p$-adically closed field $M$, making $M$ a 
topological field (usually not locally compact).  Any definable set
$X\sq M^n$ has a topological dimension, denoted by $\dim(X)$, which is
the maximal $k\leq n$ such that the image of the projection $\pi :
X\ra M^n$; $(x_1, \ldots, x_n) \mapsto (x_{r_1} , \ldots, x_{r_k} )$ has
interior, for suitable $1 \leq r_1 < \cdots < r_k \leq n$. As model
theoretic algebraic closure coincides with the field-theoretic
algebraic closure, algebraic closure gives a pregeometry on $M$, and
the algebraic dimension $\dim_{\alg}(X)$ of $X$ can be calculated in
the usual way. The topological dimension coincides with the algebraic
dimension.

\subsection{Outline}
In Section~\ref{sec:dc}, we review the notion of definable
compactness, and how it behaves in definable manifolds and definable
groups in $p$CF.  In Section~\ref{sec:dp-rank-review} we review the theory of dp-rank, which is used in Section~\ref{sec:gaps}.  In Section~\ref{sec:gaps}, we prove a technical
statement about ``gaps'' in unbounded sets, which replaces the use of
connectedness in
Peterzil-Steinhorn \cite[Lemma 3.8]{Y.-Peterzil-and-C.-Steinhorn}.
In Section~\ref{stab-review}, we review the theory of stabilizers and
$\mu$-stabilizers from \cite{Y.-Peterzil-and-S.-Starchenko}.  Finally,
we prove the main theorems in Section~\ref{sec:main}.

\section{Definable compactness} \label{sec:dc}
In this section, we review the notion of \emph{definable compactness}
for definable manifolds and definable groups in $p$-adically closed
fields.  The treatment of ($p$-adic) definable compactness in the
literature is questionable, so we build up the theory from scratch,
out of an abundance of caution.

In~Section \ref{adc} we recall an abstract definition of definable
compactness, which behaves well in any definable topological space.
In the next two sections, we restrict our attention to $p$-adic
definable manifolds.  In~Section \ref{f-def} we show that our
definition agrees with the definition in the literature in terms of
curve completion.  In Section \ref{s-def} we give another
characterization using specialization of definable types.  Finally, in
Section~\ref{dc-gr} we list some consequences for definable groups.

\subsection{Abstract definable compactness} \label{adc}
Let $M$ be an arbitrary structure.
A \emph{definable topology} on a definable set $X \subseteq M^n$ is a
topology with a (uniformly) definable basis of opens.  A \emph{definable
topological space} is a definable set with a definable topology.
 
Recall that a topological space is compact if any filtered
intersection of non-empty closed sets is non-empty.
\begin{Def}\label{defdefc}
  Let $X$ be a definable topological space in a structure $M$.  Say
  that $X$ is \emph{definably compact} if the following holds: for any
  definable family $\mathcal{F} = \{Y_t : t \in T\}$ of non-empty
  closed sets $Y_t \subseteq X$, if $\mathcal{F}$ is downwards
  directed, then $\bigcap \mathcal{F} \ne \emptyset$.

  More generally, say that a definable set $Y \subseteq X$ is definably compact
  if it is definably compact with respect to the induced subspace topology.
\end{Def}
Definable compactness has many of the expected properties:
\begin{Fact} \label{defc} ~
\begin{enumerate}
\item \label{genuine} If $X$ is a compact definable topological space, then $X$ is
  definably compact.
\item If $X, Y$ are definably compact, then $X \times Y$ is definably
  compact.
\item If $f : X \to Y$ is definable and continuous, and $X$ is
  definably compact, then the image $f(X) \subseteq Y$ is definably
  compact.
\item \label{compact-closed} If $X$ is a Hausdorff definable topological space and $Y
  \subseteq X$ is definably compact, then $Y$ is closed.
\item \label{closed-compact} If $X$ is definably compact and $Y \subseteq X$ is closed and
  definable, then $Y$ is definably compact.
\item If $X$ is a definable topological space and $Y_1, Y_2 \subseteq X$ are definably compact, then $Y_1 \cup Y_2$ is definably compact.
\end{enumerate}
\end{Fact}
Definition~\ref{defdefc} and Fact~\ref{defc} are due independently to
Fornasiero \cite{fornasiero} and the first author
\cite[Section~3.1]{wj-o-minimal}.

\begin{Rmk}\label{defc-ee}
  Suppose $X$ is a definable topological space in a structure $M$, and
  $N \succ M$.  Then $X(N)$ is naturally a definable topological space
  in the structure $N$, and $X(N)$ is definably compact if and only if
  $X$ is definably compact.  In other words, definable compactness is
  invariant in elementary extensions.
\end{Rmk}

\subsection{Definable compactness and definable manifolds in $p$CF} \label{f-def}
Let $M$ be a $p$-adically closed field with valuation group $\Gamma_M$.  Each power $M^n$ is a
definable topological space.  We first characterize definable
compactness for subsets of $M^n$.

\begin{Lemma}\label{k-cb}
  If $X \subseteq M^n$ is definably compact, then $X$ is closed and
  bounded.
\end{Lemma}
\begin{proof}
  For $t \in M \setminus \{0\}$, let $O_t$ be the $n$-dimensional ball
  $\B(0,v(t))^n$.  Each $O_t$ is clopen in $M^n$.  Therefore $\{X
  \setminus O_t : t \in M \setminus \{0\}\}$ is a downwards-directed definable
  family of closed subsets of $X$, with empty intersection.  By
  definable compactness, there is some $t$ such that $X \setminus O_t
  = \emptyset$, or equivalently, $X \subseteq O_t$.  Then $X$ is
  bounded.

  Closedness follows similarly, or by
  Fact~\ref{defc}(\ref{compact-closed}).
\end{proof}
\begin{Lemma}\label{cb-k}
  If $X \subseteq M^n$ is closed and bounded, then $X$ is definably
  compact.
\end{Lemma}
\begin{proof}
  Equivalently, if $\{Y_t\}$ is a downwards-directed definable family 
  of non-empty, closed, bounded sets, then $\bigcap_t Y_t \ne 
  \emptyset$.  This claim can be expressed as a countable conjunction 
  of $\la$-sentences.  (We need infinitely many sentences because 
  there is no bound on the complexity of the definable family 
  $\{Y_t\}$.)  As a countable conjunction of $\la$-sentences, the 
  claim holds in $M$ if and only if it holds in $\Q$.  Therefore, we may assume 
  that $M = \Q$.  In this 
  case, the set $X$ will be compact, and hence definably compact by
  Fact~\ref{defc}(\ref{genuine}).
\end{proof}

\begin{Def}
  Let $X$ be a definable topological space.  A
  \emph{$\Gamma$-exhaustion} is a definable family $\{W_\gamma \mid
  \gamma \in \Gamma_M\}$ such that
  \begin{itemize}
  \item Each $W_\gamma$ is an open, definably compact subset of $X$.
    In particular, $W_\gamma$ is clopen.
  \item If $\gamma \le \gamma'$, then $W_\gamma \subseteq
    W_{\gamma'}$.
  \item $X = \bigcup_{\gamma \in \Gamma_M} W_\gamma$.
  \end{itemize}
\end{Def}

\begin{Lemma}\label{xh-1}
  If $U \subseteq M^n$ is definable and open, then $U$ has a
  $\Gamma$-exhaustion.
\end{Lemma}
\begin{proof}
  For any $\bar{x} = (x_1,\ldots,x_n) \in M^n$ and $\gamma \in
  \Gamma_M$, let $\B(\bar{x},\gamma)$ denote the ball of valuative
  radius $\gamma$ around $\bar{x}$, i.e., $\prod_{i = 1}^n
  \B(x_i,\gamma)$.

  Let $W_{\gamma}$ be the set of $x \in U$ such that $\B(x,\gamma)
  \subseteq U$ and $\bar{0} \in \B(x,-\gamma)$.  We claim that the
  family $W_{\gamma}$ is a $\Gamma$-exhaustion.

  First of all, for all $x'$ sufficiently close to $x$, we have
  $\B(x,\gamma) = \B(x',\gamma)$ and $\B(x,-\gamma) = \B(x',-\gamma)$, 
  and so $x \in W_\gamma \iff x' \in W_\gamma$.  Therefore
  $W_\gamma$ is clopen.  Additionally,
  \begin{equation*}
    x \in W_\gamma \implies \bar{0} \in \B(x,-\gamma) \iff x \in
    \B(\bar{0},-\gamma).
  \end{equation*}
  Therefore $W_\gamma$ is bounded.  By Lemma~\ref{cb-k}, $W_\gamma$ is
  definably compact.

  If $\gamma' \ge \gamma$, then $\B(x,\gamma') \subseteq
  \B(x,\gamma)$ and $\B(x,-\gamma') \supseteq \B(x,-\gamma)$.
  Therefore
  \begin{equation*}
    x \in W_\gamma \implies x \in W_{\gamma'},
  \end{equation*}
  and the family $\{W_\gamma\}$ is monotone.

  Lastly, if $x \in U$, then for sufficiently large $\gamma$, we have
  $\B(x,\gamma) \subseteq U$, because $U$ is open.  Also, $\bar{0} \in
  \B(x,-\gamma)$ for sufficiently large $\gamma$.  Thus $x \in
  W_\gamma$ for all sufficiently large $\gamma$.  This shows $U =
  \bigcup_\gamma W_\gamma$.
\end{proof}

An \emph{$n$-dimensional definable manifold} over $M$ is a Hausdorff
definable topological space $X$ with a covering by finitely may open
subsets $U_1$,\ldots,$U_m$, and a definable homeomorphism from $U_i$ to
an open set $V_i \sq M^n$ for each $i$.

\begin{Prop} \label{xh-2}
  Let $X$ be a definable manifold in $M$.  Then $X$ has a $\Gamma$-exhaustion.
\end{Prop}

\begin{proof}
  Cover $X$ with finitely many open sets $U_i$ homeomorphic
  to open subsets of $M^n$.  For each $i$, let
  $\{W_{i,\gamma}\}_{\gamma \in \Gamma_M}$ be a $\Gamma$-exhaustion of
  $U_i$.  Let $V_\gamma = \bigcup_i W_{i,\gamma}$.  Then the family
  $\{V_\gamma\}$ is a $\Gamma$-exhaustion of $X$.
\end{proof}

\begin{Def}\label{def-bounded}
  Let $X$ be a definable manifold.  An arbitrary subset $Y \subseteq X$ is \emph{bounded} if $Y \subseteq D$ for some definably compact subset $D \subseteq X$.
\end{Def}
Proposition~\ref{prop-bounded}(\ref{pb-1}) gives a more concrete definition of ``bounded'' in terms of $\Gamma$-exhaustions.
\begin{Prop}\label{prop-bounded}
  Let $X$ be a definable manifold and $Y \subseteq X$ be an arbitrary subset.
  \begin{enumerate}
      \item \label{pb-1} Let $\{W_\gamma\}$ be a $\Gamma$-exhaustion of $X$.  Then $Y$ is bounded if and only if there is $\gamma \in \Gamma$ such that $Y \subseteq W_\gamma$.
      \item \label{pb-2} Suppose $Y$ is definable.  Then $Y$ is definably compact if and only if $Y$ is closed and bounded.
      \item \label{pb-3} Suppose $Y$ is definable.  Then $Y$ is bounded if and only if the closure $\overline{Y}$ is definably compact.
  \end{enumerate}
\end{Prop}
\begin{proof}
  ~
  \begin{enumerate}
      \item If $Y \subseteq W_\gamma$, then $Y$ is contained in the definably compact set $W_\gamma$.  Conversely, suppose $Y$ is bounded, witnessed by a definably compact set $Z \subseteq X$ with $Y \subseteq Z$.  The filtered intersection
      \begin{equation*}
      \bigcap_\gamma (Z \setminus W_\gamma)
      \end{equation*}
      is empty, so there is some $\gamma$ such that $W_\gamma \supseteq Z \supseteq Y$.
      \item If $Y$ is definably compact, then $Y$ is closed (Fact~\ref{defc}(\ref{compact-closed}), and $Y$ is bounded because $Y \subseteq Y$.  Conversely, suppose that $Y$ is closed and bounded.  Then $Y$ is a definable closed subset of a definably compact set, so $Y$ is definably compact by Fact~\ref{defc}(\ref{closed-compact}).
      \item If $\overline{Y}$ is definably compact, then $Y$ is bounded because $Y \subseteq \overline{Y}$.  Conversely, suppose that $Y$ is bounded.  Then $Y \subseteq Z$ for some definably compact set $Z \subseteq X$.  The cosure $\overline{Y}$ is a definable closed subset of $Z$, so $\overline{Y}$ is definably compact by Fact~\ref{defc}(\ref{closed-compact}). \qedhere
  \end{enumerate}
\end{proof}

%

\begin{Rmk} \label{sketchy1}
  Definable compactness is a definable property:
  Let $X_t$ be a definable manifold depending definably on some
  parameter $t \in T$.  Then
  \begin{equation*}
    \{t \in T : X_t \text{ is definably compact}\}
  \end{equation*}
  is definable.  This can be proved from Proposition~\ref{prop-bounded}(\ref{pb-1},\ref{pb-2})
  by
  compactness, using Remark~\ref{defc-ee} to reduce to the case where
  $M$ is highly saturated. 
\end{Rmk}

\begin{Rmk}
  When $M = \Q$, a definable manifold $X$ is definably compact if and only
  if it is compact.  One direction is Fact~\ref{defc}(\ref{genuine}).
  Conversely, suppose $X$ is definably compact.  Cover $X$ by
  definable open subsets $U_1, \ldots, U_n$, each homeomorphic to an
  open subset of $M^n$.  As in the proof of Proposition~\ref{xh-2}, let
  $\{W_{i,\gamma}\}_{\gamma \in \mathbb{Z}}$ be a $\Gamma$-exhaustion of
  $U_i$, and let $V_\gamma = \bigcup_{i = 1}^n W_{i,\gamma}$, so that
  $\{V_\gamma\}_{\gamma \in \mathbb{Z}}$ is a $\Gamma$-exhaustion of $X$.  By
  Proposition~\ref{prop-bounded},
  there is some $\gamma \in \mathbb{Z}$ such that $X
  = V_\gamma$.  Then $X = \bigcup_{i = 1}^n W_{i,\gamma}$, where each
  $W_{i,\gamma}$ is definably compact.  Lemmas~\ref{k-cb} and
  \ref{cb-k} imply that definable compactness is equivalent to
  compactness for definable subsets of $M^n$.  Therefore each
  $W_{i,\gamma}$ is compact.  As $X$ is covered by finitely many
  compact sets, $X$ itself is compact.
\end{Rmk}

We now try to relate our notion of definable compactness to the more
familiar notions appearing in \cite{O-P}.

\begin{Def}
  Let $X$ be a definable manifold.  Let $D$ be a definable subset of
  $M \setminus \{0\}$ with $0 \in \overline{D}$.  Let $f : D \to X$ be
  a definable function.  Then $a \in X$ is a \emph{cluster point} of
  $f$ if $(0,a)$ is in the closure of the graph of $f$.  In other
  words, for every neighborhood $U_1$ of 0 and every neighborhood
  $U_2$ of $a$, there is $x \in U_1 \cap D$ such that $f(x) \in U_2$.
\end{Def}

\begin{Lemma} \label{stupid-lemma}
  let $X$ be a definable manifold.  Let $f : R(M) \setminus \{0\} \to
  X$ be a definable function.  Then $f$ is continuous at all but
  finitely many points of $R(M)$.
\end{Lemma}
\begin{proof}
  An exercise using the fact that any definable function $M \to M^n$
  is continuous off a finite set.
\end{proof}

\begin{Lemma}\label{defs-agree}
  Let $X$ be a definable manifold.  Let $Y$ be a definable subset.
  The following are equivalent:
  \begin{enumerate}
  \item $Y$ is definably compact.
  \item \label{defs-agree-strange} If $D$ is a definable subset of $M \setminus \{0\}$ with $0
    \in \overline{D}$, then every definable function $f : D \to Y$ has
    a cluster point.
  \item \label{da-3} Any definable continuous function $f : R(M) \setminus \{0\}
    \to Y$ has a cluster point in $Y$.
  \item Any definable continuous function $f : \B(0,\gamma) \setminus
    \{0\} \to Y$ has a cluster point in $Y$.
  \item Let $\{Z_\gamma\}_{\gamma \in \Gamma_M}$ be a definable family of non-empty closed subsets of $Y$, such that $\gamma \le \gamma' \implies Z_{\gamma} \supseteq Z_{\gamma'}$.  Then $\bigcap_{\gamma \in \Gamma_M} Z_\gamma \ne \emptyset$.
  \end{enumerate}
\end{Lemma}
\begin{proof}
  (1)$\Rightarrow$(2): the set of cluster points is the intersection
  \begin{equation*}
    \bigcap_{\gamma \in \Gamma_M} \overline{f(\B(0,\gamma) \cap D)}.
  \end{equation*}
  This is non-empty by definable compactness of $Y$.

  (2)$\Rightarrow$(3) is trivial, and (3)$\Rightarrow$(4) follows by rescaling.

  (4)$\Rightarrow$(5): By definable Skolem functions, there is some definable function $f : M \setminus \{0\} \to Y$ such that $f(x) \in Z_{v(x)}$ for all $x \in M \setminus \{0\}$.  By Lemma~\ref{stupid-lemma}, there is some $\delta \in \Gamma_M$ such that $f$ is continuous on $\B(0,\delta) \setminus \{0\}$.  By (4), $f$ has a cluster point $a \in Y$.  Then $a \in \bigcap_{\gamma} Z_\gamma$.  Otherwise, take $\gamma$ large enough that $a \notin Z_\gamma$.  Because $a$ is a cluster point and $Z_\gamma$ is closed in $Y$, there is some $x \ne 0$ such that $v(x) \ge \gamma$ and $f(x) \notin Z_\gamma$.  By choice of $f$, $f(x) \in Z_{v(x)} \subseteq Z_\gamma$, a contradiction.
  
  (5)$\Rightarrow$(1):  We first claim that $Y$ is closed.  Take $p \in \overline{Y}$.  Because $X$ is a definable manifold, we can identify a neighborhood of $p$ in $X$ with the closed ball $R(M)^n$ in $M^n$.  For $\gamma \ge 0$, let $B_\gamma$ be the closed ball of radius $\gamma$ around $p$.  For $\gamma \le 0$ let $B_\gamma = B_0$.  Then $B_\gamma \cap Y$ is a non-empty closed subset of $Y$ for any $\gamma$, because $p \in \overline{Y}$.  By (4), the intersection $\bigcap_{\gamma} (B_\gamma \cap Y)$ is non-empty, and so $p \in Y$.  Therefore $Y$ is closed.
  
  Similarly, $Y$ is bounded.  Take a $\Gamma$-exhaustion $\{U_\gamma\}_{\gamma \in \Gamma_M}$ of the definable manifold $X$.  If $Y$ is unbounded, then $Y \setminus U_\gamma$ is a closed non-empty subset of $Y$ for each $\gamma$.  Applying (5) to the family of sets $Y \setminus U_\gamma$, we see that $Y \not \subseteq \bigcup_{\gamma} U_\gamma = X$, a contradiction.  Therefore $Y$ is closed and bounded.  By Proposition~\ref{prop-bounded}(\ref{pb-2}), $Y$ is definably compact. 
\end{proof}

Therefore,  we could alternatively define
definable compactness as follows:
\begin{Prop}\label{fixed-def}
  
  Let $Y$ be a definable subset of a definable manifold $X$.  Then $Y$ is definably compact if and only if every 
  definable continuous function $f : R(M) \setminus \{0\} \to Y$ has a
  cluster point.
\end{Prop}
This is essentially the definition of ``definable compactness''
appearing in \cite{O-P} (with the mistake fixed).

\subsection{Definable compactness and definable 1-dimensional types} \label{s-def}

Suppose that $N \succ M$.  Let $X$ be a definable manifold in $M$.
\begin{Def}
  For $a \in X(M)$ and $b \in X(N)$, say that $a$ and $b$ are
  \emph{infinitesimally close} over $M$ if $b$ is contained in every
  $M$-definable neighborhood of $a$.
\end{Def}
Suppose that $X, Y$ are $M$-definable manifolds and $f : X \to Y$ is
an $M$-definable continuous function.  If $a \in X(M)$ is
infinitesimally close to $b \in X(N)$, then $f(a)$ is infinitesimally
close to $f(b)$.
\begin{Def} ~ \label{o-st-def}
  \begin{itemize}
  \item We let $\mathcal{O}_{X(M)}(N)$ denote the set of $b \in X(N)$
    such that $b$ is infinitesimally close to at least one $a \in
    X(M)$.
  \item There is a function $\st_M^N : \mathcal{O}_{X(M)}(N) \to X(M)$
    sending each $b$ to the unique $a \in X(M)$ such that $b$ and $a$
    are infinitesimally close.  This is well-defined because $X$ is
    Hausdorff.
  \end{itemize}
\end{Def}
The map $\st_M^N$ is the ``standard part'' map from
$\mathcal{O}_{X(M)}(N)$ to $X(M)$.
\begin{Def}
  If $p$ is a complete $X$-type over $M$, we say that $p$
  \emph{specializes to} $a \in X(M)$ if $p(x) \vdash x \in U$ for
  every $M$-definable neighborhood $U \ni a$.
\end{Def}
If $b \in X(N)$ is a realization of $p$, then $p$ specializes to $a$
if and only if $\st_M^N(b) = a$.

\begin{Fact}\label{infinitesimal type}
If $a'\in N\backslash M$ is infinitesimally close to $ a\in M$ over $M$, then there is a coset $C\sq N\backslash\{0\}$ of $\bigcap_{n\geq 1}P_n(N)$ such that $\tp(a'/M)$ is determined by the partial type
\[
\{v(x-a)>\gamma \mid \gamma\in \Gamma_M\}\cup \{x-a\in C\},
\]
and $\tp(a'/M)$ is definable over $M$.
\end{Fact}
This follows by a similar argument to Lemma 2.1 in \cite{GPY}.

\begin{Lemma} \label{chains}
  Let $\mathcal{C}$ be a definable (i.e., interpretable) family of
  balls $B \subseteq M$.  Suppose the following conditions hold:
  \begin{enumerate}
  \item $\mathcal{C}$ is non-empty.
  \item $\mathcal{C}$ is a chain: it is linearly ordered by $\subseteq$.
  \item $\mathcal{C}$ is upwards-closed: if $B \supseteq B' \in
    \mathcal{C}$ for balls $B, B'$, then $B \in \mathcal{C}$.
  \item $\mathcal{C}$ has no minimal element.
  \end{enumerate}
  Then there is $d \in M$ such that $\mathcal{C}$ is the set of balls
  containing $d$.
\end{Lemma}
\begin{proof}
  We may assume $M = \Q$, in which case the lemma is an easy exercise 
  using spherical completeness of $\Q$.
\end{proof}

\begin{Lemma} \label{1-d-types}
  Let $X$ be an $M$-definable set, and $p$ be a 1-dimensional
  definable type over $M$ in $X$.  Then there is an elementary
  extension $N \succ M$ and elements $a \in M$, $b \in X(M)$, such
  that $a$ is infinitesimally close to 0, $b \in \dcl(Ma)$, and $p =
  \tp(b/M)$.
\end{Lemma}
\begin{proof}
  Take $N \succ M$ containing a realization $b$ of $p$.  Because $p$
  is 1-dimensional, there is some singleton $c \in N$ such that
  $\dcl(Mb) = \dcl(Mc)$.  (In fact, we can take $c$ to be a coordinate
  of the tuple $b$.)  Replacing $c$ with $1/c$ if necessary, we may assume that 
  $v(c) \ge 0$.  Then $\tp(c/M)$ is definable and one-dimensional.
  Let $\mathcal{C}$ be the family of $M$-definable balls which contain
  $c$.  Then $\mathcal{C}$ is definable,
  because $\tp(c/M)$ is definable.  Moreover, $\mathcal{C}$ satisfies
  the four conditions of Lemma~\ref{chains}:
  \begin{enumerate}
  \item $\mathcal{C}$ is non-empty, because it contains the ball $R(M)$ of
    radius 0.
  \item $\mathcal{C}$ is a chain, because any two balls which
    intersect are comparable, and $\mathcal{C}$ cannot contain two
    disjoint balls.
  \item $\mathcal{C}$ is upwards-closed, trivially.
  \item $\mathcal{C}$ has no least element.  Otherwise, if $B$ were
    the smallest $M$-definable ball containing $c$, then we could
    write $B$ as a disjoint union of smaller balls $B = B_1 \cup
    \cdots \cup B_p$, and one of the $B_i$ would belong to
    $\mathcal{C}$.
  \end{enumerate}
  By Lemma~\ref{chains}, $\mathcal{C}$ is the class of balls
  around some point $d$.  So there is some $d \in M$ such
  that $c$ is contained in every $M$-definable ball around $d$.
  Therefore, $c$ is infinitesimally close to $d$ over $M$.  Take $a =
  c - d$.
\end{proof}

\begin{Lemma}\label{specializing-types}
  Let $X$ be a definable manifold over $M$.  Let $Y$ be a definably
  compact definable subset of $X$.  Let $p$ be a definable
  1-dimensional complete $Y$-type over $M$.  Then $p$ specializes to a
  point in $Y$.
\end{Lemma}
\begin{proof}
  Let $N$ be an $\aleph_1$-saturated elementary extension of $M$, and
  let $\M$ be a monster model extending $N$.  Let $p^N$ be the
  heir of $p$ over $N$.  We first show that 
  $p^N$ specializes to a point in $Y(N)$.  
  Take $c \in Y(\M)$
  realizing $p^N$.  By Lemma~\ref{1-d-types}, we can write $c$ as
  $g(a)$ for some $N$-definable function $g : \M \to Y(\M)$ and some $a \in
  \M$ infinitesimally close to 0 over $N$.  Because $N$ is
  $\aleph_1$-saturated, there is some $u \in N$ such that $a/u \in
  P_n(\M)$ for all $n$.  Replacing $a$ with $a/u$, we may assume that $a
  \in P_n(\M)$ for all $n$.  For each $n$, let $S_n \subseteq Y(N)$ be the
  definable set of cluster points of $g \restriction P_{n}(N)$.  Each
  $S_n$ is closed, and non-empty by Lemma~\ref{defs-agree}(\ref{defs-agree-strange}).  The
  intersection $\bigcap_n S_n$ is filtered, and therefore non-empty by
  $\aleph_1$-saturation.  Take $b \in \bigcap_n S_n$.  Let $\Sigma(x)$ be
  the partial type saying that $x$ is infinitesimally close to 0, $g(x)$ is
  infinitesimally close to $b$, and $x \in P_n$ for all $n$.  Then
  $\Sigma(x)$ is finitely satisfiable, by choice of $b$.  Take $a' \in \M$
  realizing $\Sigma(x)$.  By Fact~\ref{infinitesimal type}, $\tp(a'/N) =
  \tp(a/N)$.  Therefore $a$ satisfies $\Sigma(x)$, and so $g(a)$ is
  infinitesimally close to $b$.  It follows that $p^N(x)$ specializes to
  $b$.

  Let $Z$ be the set of $b \in Y(N)$ such that $p^N$ specializes to
  $b$.  The set $Z$ is $M$-definable, because $p^N$ is definable over
  $M$.  The above argument shows $|Z| > 0$.  On the other hand, $|Z|
  \le 1$ because $Y(N)$ is Hausdorff.  Therefore $Z$ is a singleton
  $\{b\}$, and the element $b$ lies in $Y(M)$.  Then $p$ specializes
  to $b$.
\end{proof}

\begin{Prop}\label{sp-char}
  Work in a model $M$.  Let $X$ be a definable manifold and $Y$ be a
  definable subset.  Then $Y$ is definably compact if and only if
  every 1-dimensional definable $Y$-type specializes to a point of
  $Y$.
\end{Prop}
\begin{proof}
  One direction is Lemma~\ref{specializing-types}.  Conversely,
  suppose every 1-dimensional definable type in $Y$ specializes to a
  point.  We claim that $Y$ is definably compact.  We use criterion
  (\ref{da-3}) of Lemma~\ref{defs-agree}.  Let $f : R(M)
  \setminus \{0\} \to Y$ be a definable continuous function.  Take a
  monster model $\M \succ M$ and a non-zero $a \in \M$ infinitesimally
  close to 0 over $M$.  Let $b = f(a)$.  By Fact~\ref{infinitesimal
    type}, $\tp(a/M)$ is definable.  Therefore $\tp(b/M)$ is
  1-dimensional and definable.  Then $\tp(b/M)$ specializes to a point
  $c \in Y(M)$.  We claim that $c$ is a cluster point of $f$.  For any
  $M$-definable neighborhoods $U_1 \ni 0$ and $U_2 \ni c$, we have
  $(a,f(a)) = (a,b) \in U_1 \times U_2$.  As $M \prec \M$, there must
  be some $(a',f(a')) \in U_1(M) \times U_2(M)$.  This shows that $c$ is a
  cluster point of $f$.
\end{proof}

\begin{Lemma}\label{crux}
  Let $X$ be an $M$-definable manifold and $\{O_t\}_{t \in \Gamma_M}$
  be a $\Gamma$-exhaustion.  Let $p$ be a definable 1-dimensional type
  in $X$ over $M$, such that $p$ does not concentrate on $O_t$ for any
  $t \in \Gamma_M$.  Suppose $\M \succ N \succ M$.  Suppose that $b
  \in X(\M)$ realizes $p$, and $b \notin O_t(\M)$ for any $t \in
  \Gamma_N$.  Then $b$ realizes $p^N$, the heir of $p$
  over $N$.
\end{Lemma}
\begin{proof}
  By Lemma~\ref{1-d-types}, we have $b = f(a)$ for some $M$-definable
  function $f : \M \to X$ and some $a \in \M$ infinitesimally close to
  0 over $M$.  By Lemma~\ref{stupid-lemma}, $f$ is continuous on
  $\B(0,\gamma_0)$ for some sufficiently large $\gamma_0 \in
  \Gamma_M$; note that $v(a) > \gamma_0$.  We claim that $a$ is
  infinitesimally close to 0 over $N$.  Otherwise, there is some
  $\gamma \in \Gamma_N$ such that $v(a) < \gamma$.  Let $A$ be the
  definable set of $x \in \M$ such that $\gamma_0 < v(x) < \gamma$;
  note that $a \in A$.  The set $A$ is definably compact and
  $N$-definable.  Also, $f$ is $N$-definable and continuous on $A$.
  Therefore, the image $f(A)$ is $N$-definable, and definably compact.
  By Proposition~\ref{prop-bounded}, 
  there is some $t \in \Gamma_N$ such that
  $f(A) \subseteq O_t$.  Then $b = f(a) \in f(A) \subseteq O_t(\M)$,
  contradicting the assumptions.

  This shows that $a$ is infinitesimally close to 0 over $N$.  By
  Fact~\ref{infinitesimal type}, $\tp(a/N)$ is the heir of $\tp(a/M)$,
  implying that $\tp(b/N) = \tp(f(a)/N)$ is the heir of $\tp(f(a)/M) =
  p$.
\end{proof}

\subsection{Definable groups in $p$CF} \label{dc-gr}

By a \emph{definable group} over $M$, we mean a definable set with a
definable group operation. By \cite{Pillay-G-in-p}, any group $G$
definable in $M$ admits a unique definable manifold structure making
the group operations be continuous.

\begin{Rmk}
  \label{sketchy2}
  In particular, there is a canonical notion of ``definable
  compactness'' for abstract definable groups and their definable
  subsets.  As in Remarks~\ref{defc-ee} and \ref{sketchy1}, one can
  show that these notions are definable in families and invariant in
  elementary extensions.
\end{Rmk}

\begin{Def}\label{def-gnb}
  A \emph{good neighborhood basis} is a definable neighborhood basis
  of the form $\{O_t : t \in \Gamma_M\}$ which is also a
  $\Gamma$-exhaustion, and such that $O_t = O_t^{-1}$ for each $t \in
  \Gamma_M$.
\end{Def}

\begin{Prop}
  Every definable group has a good neighborhood basis.
\end{Prop}
\begin{proof}
  By Proposition~\ref{xh-2}, the group $G$ admits a $\Gamma$-exhaustion
  $\{W_t : t \in \Gamma_M\}$.  Replacing $\{W_t\}$ with $\{W_{t +
    \gamma}\}$, we may assume that $W_0$ is non-empty.  Replacing
  $\{W_t\}$ with $\{a \cdot W_t\}$, we may assume that $\id_G \in
  W_0$.

  Because $G$ is a definable manifold, there is some
  definable neighborhood basis $\{N_t : t \in \Gamma_M,\ t < 0\}$ such
  that each $N_t$ is clopen, and $N_t$ depends monotonically on $t$.
  Define
  \begin{equation*}
    B_t =
    \begin{cases}
      W_t & t \ge 0 \\ W_0 \cap N_t & t < 0.
    \end{cases}
  \end{equation*}
  Then $\{B_t : t \in \Gamma_M\}$ is a definable neighborhood basis
  and a $\Gamma$-exhaustion.  Lastly, define $O_t = B_t \cap
  B_t^{-1}$.  Then $\{O_t : t \in \Gamma_M\}$ has all the desired
  properties.
\end{proof}

\begin{Prop} \label{up-down}
  Let $\{O_t : t \in \Gamma_M\}$ be a good neighborhood basis of a definable group $G$.
  \begin{enumerate}
  \item For any $t \in \Gamma_M$, there is $t' \in \Gamma_M$ such that
    $O_{t'} \cdot O_{t'} \subseteq O_t$.
  \item For any $t \in \Gamma_M$, there is $t'' \in \Gamma_M$ such that
    $O_t \cdot O_t \subseteq O_{t''}$.
  \end{enumerate}
\end{Prop}

\begin{proof}
  (1) is by continuity.  For (2), note that the set $O_t \cdot O_t$ is
  an image of the definably compact space $O_t \times O_t$ under the
  definable continuous map $(x,y) \mapsto x \cdot y$.  Therefore $O_t
  \cdot O_t$ is definably compact.  Then $t''$ exists by
  Proposition~\ref{prop-bounded}.
\end{proof}

\begin{Lemma} \label{conjugation}
  Let $\{O_t : t \in \Gamma_M\}$ be a good neighborhood basis of a
  definable group $G$.  For every $t, \epsilon \in \Gamma_M$, there is
  $\delta \in \Gamma_M$ such that if $a \in O_\delta$ and $b \in O_t$,
  then $b^{-1} a b \in O_\epsilon$.
\end{Lemma}
\begin{proof}
  Define $S_\delta = \{(a,b) \in O_\delta \times O_t : b^{-1} a b
  \notin O_\epsilon\}$.  Suppose for the sake of contradiction that
  $S_\delta \ne \emptyset$ for all $\delta$.  The family $S_\delta$ is
  definable, and depends monotonically on $\delta$.  Each set
  $S_\delta$ is closed, because $O_\epsilon, O_\delta$, and $O_t$ are
  clopen.  By definable compactness of $O_{\delta} \times O_t$, the
  intersection $\bigcap_\delta S_\delta$ is non-empty.  Therefore
  there are $a, b \in G$ such that
  \begin{enumerate}
  \item $a \in O_\delta$ for all $\delta$.
  \item $b \in O_t$.
  \item $b^{-1} a b \notin O_\epsilon$.
  \end{enumerate}
  The first point implies $a = \id_G$, which then implies $b^{-1}
  a b = \id_G \in O_\epsilon$, a contradiction.
\end{proof}

\section{Review of dp-rank} \label{sec:dp-rank-review}

In Section~\ref{sec:gaps} we will make extensive use of dp-rank, so we review its basic properties here.  
\begin{Def}
  Let $\kappa$ be a cardinal and $\Sigma(x)$ be a partial type.  An \emph{ict-pattern} of depth $\kappa$ in $\Sigma(x)$ consists of
  \begin{itemize}
      \item A family of formulas $\{\phi_\alpha(x,y_\alpha)\}_{\alpha < \kappa}$.
      \item An array of parameters $\{b_{\alpha,i}\}_{\alpha < \kappa,~i < \omega}$ with $|b_{\alpha,i}| = |y_\alpha|$.
  \end{itemize}
  such that for any function $\eta : \kappa \to \omega$, the following type is consistent:
  \begin{equation*}
      \Sigma(x) \cup \{\phi_\alpha(x,b_{\alpha,i}) : \alpha < \kappa, ~ i = \eta(\alpha)\} \cup \{\neg \phi_\alpha(x,b_{\alpha,i}) : \alpha < \kappa, ~ i \ne \eta(\alpha)\}.
  \end{equation*}
\end{Def}
\begin{Def}
  The \emph{dp-rank} of a partial type $\Sigma(x)$ is the supremum of cardinals $\kappa$ such that, in some elementary extension $N \succeq M$, there is an ict-pattern of depth $\kappa$ in $\Sigma(x)$.  When there is no supremum, the dp-rank is defined to be $\infty$, a formal symbol greater than all cardinals.
  
  We write the dp-rank of $\Sigma(x)$ as $\dpr(\Sigma)$.  When $\Sigma(x)$ is a complete type $\tp(b/A)$, we write the dp-rank as $\dpr(b/A)$.
\end{Def}
The following facts can be found in \cite{dp-add}, or alternatively \cite[Chapter 4]{NIPguide}.
\begin{Fact}\label{dpr1}
  The following are equivalent in a structure $M$:
  \begin{enumerate}
      \item $M$ is NIP.
      \item $\dpr(x=x) < \infty$.
      \item Every partial type has dp-rank $< \infty$.
  \end{enumerate}
\end{Fact}
\begin{Fact}
  If $\Sigma(x)$ is a partial type over $A$, and if the ambient model $M$ is $|A|^+$-saturated, then $\dpr(\Sigma)$ is the supremum of $\dpr(b/A)$ as $b$ ranges over realizations of $\Sigma(x)$.
\end{Fact}
\begin{Fact}
  If $b \in \acl(A)$, then $\dpr(b/A) = 0$.  If $b \notin \acl(A)$, then $\dpr(b/A) > 0$.
\end{Fact}
\begin{Fact} \label{dpr4}
  For any $b, c, A$, we have
  \begin{equation*}
      \dpr(b/A) \le \dpr(bc/A) \le \dpr(b/cA) + \dpr(c/A).
  \end{equation*}
\end{Fact}
It is also helpful to view dp-rank as a property of definable sets:
\begin{Def}
  If $D$ is a definable set, then the dp-rank of $D$, written $\dpr(D)$, is $\dpr(\phi(x))$ for any formula $\phi(x)$ defining $D$.
\end{Def}
The following facts are easy exercises using Facts~\ref{dpr1}--\ref{dpr4}.
\begin{Fact}
  $\dpr(D) > 0$ if and only if $D$ is infinite.
\end{Fact}
\begin{Fact}
  If $D_1, D_2$ are definable sets, then $\dpr(D_1 \times D_2) = \dpr(D_1) + \dpr(D_2)$.
\end{Fact}
\begin{Fact}
  If $f : D_1 \to D_2$ is a definable injection, then $\dpr(D_1) \le \dpr(D_2)$.  If $f : D_1 \to D_2$ is a definable surjection, then $\dpr(D_1) \ge \dpr(D_2)$.
\end{Fact}
We will need the following about dp-rank in $p$-adically closed fields:
\begin{Fact}[{\cite[Theorem~6.6]{dpExamples}}]
  If $M$ is a $p$-adically closed field, then $\dpr(M) = 1$.
\end{Fact}
\begin{Cor}
  If $M$ is a $p$-adically closed field, then every $n$-type in $M$ has dp-rank at most $n$.
\end{Cor}
In fact, dp-rank in $p$CF agrees with the natural notion of dimension (topological dimension or $\acl$-dimension), by \cite[Exercise~4.38]{NIPguide}.  We will not need this fact, however.

\section{Large gaps} \label{sec:gaps}
In order to apply the strategy of Peterzil and Steinhorn, we need a
technical statement about ``gaps'' in unbounded curves:
\begin{Conj}\label{stuck} 
  Let $G$ be a definable group over a $p$-adically closed field $M$,
  with a good neighborhood basis $\{O_t \mid t \in \Gamma_M\}$.  Let
  $I$ be a 1-dimensional unbounded definable subset of $G$.  Then for
  every $t_0 \in \Gamma_M$, there is $t \in \Gamma_M$ such that
  \begin{equation*}
    \{g \in I \mid g(O_t \setminus O_{t_0}) \cap I = \emptyset\}
  \end{equation*}
  is bounded.
\end{Conj}
The o-minimal analogue of Conjecture~\ref{stuck} holds by an easy
connectedness argument \cite[Lemma 3.8]{Y.-Peterzil-and-C.-Steinhorn}.
But in a $p$-adically closed field, everything is totally disconnected
and we need a completely different approach.  In the end, we will prove Conjecture~\ref{stuck} only in a special case (Proposition~\ref{gaps-1}), namely when $G$ is nearly abelian (Definition~\ref{d-na}).

\begin{Rmk}
  It is useful to consider what a counterexample to
  Conjecture~\ref{stuck} would look like.  For each $t \gg t_0$, there
  would be unboundedly many $g \in I$ such that
  \begin{equation*}
    g (O_t \setminus O_{t_0}) \cap I = \emptyset,
  \end{equation*}
  or equivalently $g O_t \cap I = g O_{t_0} \cap I$.  Around $g$, the
  set $I$ looks like an ``island'' $g O_{t_0} \cap I$ surrounded by a
  very large empty space $g (O_t \setminus O_{t_0})$.  Since $I$ is
  unbounded, there must be infinitely many of these ``islands.''
  Because this holds for \emph{any} $t$, the gaps between the islands
  must become greater and greater as we move towards ``$\infty$''.
\end{Rmk}

The behavior described above is reminiscent of the behavior of the set
$2^{\mathbb{Z}}$ in the group $(\R,+,<)$.  The structure
$(\R,+,<,2^{\mathbb{Z}})$ is NIP \cite[Theorem~6.5]{dep-pairs} but it
does not have finite dp-rank, and this is a direct consequence of the
``large gaps'' in $2^{\mathbb{Z}}$.  In a non-standard elementary
extension, by choosing $a_1 < b_1 < a_2 < b_2 < \cdots < a_n < b_n$
carefully, one can ensure that the map
\begin{gather*}
  \prod_{i = 1}^n (2^{\mathbb{Z}} \cap [a_i, b_i]) \to \R  \\
  (x_1,\ldots,x_n) \mapsto \sum_{i = 1}^n x_i
\end{gather*}
is injective and each set $2^{\mathbb{Z}} \cap [a_i,b_i]$ is infinite,
showing that the model has dp-rank at least $n$ (for arbitrary finite
$n$).

Our approach for attacking Conjecture~\ref{stuck} is based on this
line of argument: take a set $I$ with large gaps and obtain infinite
dp-rank.  Unfortunately, the argument only works in the nearly abelian case
(Proposition~\ref{gaps-1}), though we can salvage a much weaker
statement in the non-abelian case (Proposition~\ref{gaps-2}).

\subsection{Notation}

Let $G$ be a group.  If $H$ is a subgroup of $G$, we let $G/H$ denote the set of left cosets of $H$. 
If $A \subseteq G$, we will write $A/H$
to indicate the image of $A$ in $G/H$. 
If $A, B \subseteq G$, we let $A^B$ indicate $\{b^{-1}ab : a \in A, ~ b \in B\}$. 
Notation like ``$X \setminus Y$'' will always mean set subtraction, rather than quotienting by a group action on the left.

\begin{Def}\label{def-diamond}
  Let $X, Y$ be subsets of a group $G$.  Define
  \begin{equation*}
    X \diamond Y = \{g \in X : gY \cap X = \emptyset\}.
  \end{equation*}
\end{Def}
Note that $X \diamond Y$ depends negatively on $Y$.  We will write
``$A \diamond B \setminus C$'' to mean ``$A \diamond (B \setminus
C)$.''
\begin{Rmk} \label{gap-crossing}
  Suppose $X, Y$ are subgroups of $G$, $S\sq G$, and $a, b \in S \diamond X
  \setminus Y$.  Then
  \begin{equation*}
    aX = bX \implies aY = bY.
  \end{equation*}
  Otherwise, $b = a \delta$ for some $\delta \in X \setminus Y$, and
  so $b \in a (X \setminus Y) \cap S$, contradicting the fact that $a
  (X \setminus Y) \cap S = \emptyset$.
\end{Rmk}

\subsection{The bad gap configuration}

Recall that an \emph{externally definable} set $X$ in a structure $M$
is a set of the form $Y \cap M^n$ for some elementary extension $N
\succ M$ and definable set $Y \subseteq N^n$.  The \emph{Shelah
expansion} $M^{Sh}$ is the expansion of $M$ by all externally
definable sets.  When $M$ is NIP, the Shelah expansion $M^{Sh}$ has
elimination of quantifiers \cite[Proposition~3.23]{NIPguide}.  Using
this, it is easy to see that $M^{Sh}$ has the same dp-rank as $M$.

\begin{Rmk} \label{ext-chain}
  Let $\mathcal{F}$ be a collection of definable subsets of $M^n$.  If
  the sets in $\mathcal{F}$ are uniformly definable, and $\mathcal{F}$
  is linearly ordered by inclusion, then the sets $\bigcup
  \mathcal{F}$ and $\bigcap \mathcal{F}$ are externally definable
  \cite[Kaplan's Lemma 3.4]{hhj-v-top}.
\end{Rmk}
Later, we will use Remark~\ref{ext-chain} in conjunction with Proposition~\ref{up-down} to construct externally definable subgroups of definable groups.

\begin{Def}\label{bad-config}
  Let $G$ be a definable group in a structure $M$.  A \emph{bad gap configuration} in $G$ consists of the following
  \begin{itemize}
      \item A finite subgroup $F \subseteq G$.
      \item Externally definable subgroups
      \begin{equation*}
      \cdots \subseteq Y_2 \subseteq Y_1 \subseteq Y_0 \subseteq X_0
      \subseteq X_1 \subseteq \cdots \subseteq G
      \end{equation*}
      \item An externally definable subset $I \subseteq G$.
  \end{itemize}
  such that the following conditions hold:
  \begin{itemize}
  \item $Y_i^F \subseteq Y_i$ for all $i$.
  \item $Y_i^{X_i} \subseteq Y_{i-1}$, for $i > 0$.
  \item $(X_i \cap (I \diamond X_{i-1} \setminus F Y_{i-1}))/X_{i-1}$ is
    infinite, for $i > 0$.
  \end{itemize}
  We say that a bad gap configuration is \emph{($A$-)definable} if all of $F$, the $X_i$, $Y_i$, and $I$ are ($A$-)definable.
\end{Def}

\begin{Lemma}\label{new-config-0}
  If $G$ has finite dp-rank, then there is no bad gap configuration in $G$.
\end{Lemma}
\begin{proof}
  Let $(F,\{Y_i\},\{X_i\},I)$ be a bad gap configuration.
  Replacing $M$ with the Shelah expansion $M^{Sh}$, we may assume that
  the bad gap configuration is definable.
  Passing to an elementary
  extension and naming parameters, we may assume that $M$ is
  $\aleph_1$-saturated and the bad gap configuration is $\emptyset$-definable.

  Note that $FY_i = Y_iF$ is a subgroup of $G$, and that the index of
  $Y_i$ if $FY_i$ is finite, no more than $|F|$.  Let $D_i$ be the
  definable set $X_i \cap (I \diamond X_{i-1} \setminus F Y_{i-1})$.
  By assumption, $D_i/X_{i-1}$ is infinite.
  \begin{Claim1}  
    Suppose $a_i, a'_i \in D_i$ for $i = 1, \ldots, n$, and suppose
    \begin{equation}
      Y_n a_n a_{n-1} \cdots a_1 = Y_n a'_n a'_{n-1} \cdots a'_1. \label{eq1}
    \end{equation}
    Then $a_n F Y_{n-1} = a'_n F Y_{n-1}$.  If moreover $a_n Y_{n-1} =
    a'_n Y_{n-1}$, then
    \begin{equation}
      Y_{n-1} a_{n-1} \cdots a_1 = Y_{n-1} a'_{n-1} \cdots a'_1. \label{eq2}
    \end{equation}
  \end{Claim1}
   
  \begin{claimproof}
    Note that $a_i, a_i' \in X_i$.
    Equation (\ref{eq1}) implies that
    \begin{equation}
      a'_n a'_{n-1} \cdots a'_1 = \epsilon a_n a_{n-1} \cdots a_1 =
      a_n \epsilon^{a_n} a_{n-1} \cdots a_1 \label{eq3}
    \end{equation}
    for some $\epsilon \in Y_n$.  Then $\epsilon^{a_n} \in Y_n^{X_n}
    \subseteq Y_{n-1} \subseteq X_{n-1}$.  For $i < n$, we have $a_i,
    a'_i \in X_i \subseteq X_{n-1}$.  Therefore (\ref{eq3}) implies
    that $a'_n X_{n-1} = a_n X_{n-1}$.  Both $a'_n$ and $a_n$ are in
    $I \diamond X_{n-1} \setminus F Y_{n-1}$, so by
    Remark~\ref{gap-crossing} we have $a'_n F Y_{n-1} = a_n F Y_{n-1}$
    as desired.  Now suppose that $a_n Y_{n-1} = a'_n Y_{n-1}$.  Then
    $a'_n = a_n \delta$ for some $\delta \in Y_{n-1}$.  Then Equation
    (\ref{eq3}) implies
    \begin{gather*}
      a_n \epsilon^{a_n} a_{n-1} \cdots a_1 = a_n \delta a'_{n-1} a'_{n-2} \cdots a'_1 \\
      \epsilon^{a_n} a_{n-1} \cdots a_1 = \delta a'_{n-1} a'_{n-2} \cdots a'_1.
    \end{gather*}
    Both $\epsilon^{a_n}$ and $\delta$ are in $Y_{n-1}$, so Equation
    (\ref{eq2}) holds.
  \end{claimproof}
  For any $n \in \N$, we claim that $\dpr(G) \ge n$.
  By assumption, the
  interpretable set $D_i/X_{i-1}$ is infinite.  The interpretable set
  $D_i/Y_{i-1}$ is even bigger, because $Y_{i-1} \subseteq X_{i-1}$.
  By the properties of dp-rank in Section~\ref{sec:dp-rank-review}, $\prod_{i = 1}^n D_i/Y_{i-1}$ has dp-rank
  at least $n$.  Take a tuple $\bar{b} \in \prod_{i = 1}^n
  D_i/Y_{i-1}$ such that $\dpr(\bar{b}/\emptyset) \ge n$.  Each $b_i$
  is a coset $a_iY_{i-1}$ for some $a_i \in D_i$.  Let $c = a_na_{n-1}
  \cdots a_1 \in G$.
  \begin{Claim2}
    For each $i$, we have $b_i \in \acl(c, b_{i+1}, \ldots, b_n)$.
  \end{Claim2}
  \begin{claimproof}
    Let $S$ be the set of $(a'_1,\ldots,a'_n) \in \prod_j D_j$ such
    that
    \begin{itemize}
    \item $a'_n a'_{n-1} \cdots a'_1 = c$.
    \item $a'_j Y_{j-1} = b_j = a_j Y_{j-1}$ for $j > i$.
    \end{itemize}
    Then $(a_1,\ldots,a_n) \in S$ and $S$ is definable over $c,
    b_{i+1}, \ldots, b_n$.  If $(a'_1,\ldots,a'_n) \in S$, then
    \begin{equation*}
      Y_n a'_n a'_{n-1} \cdots a'_1 = Y_n c
      = Y_n a_n a_{n-1} \cdots a_1.
    \end{equation*}
    By Claim 1 applied $(n-i+1)$ times, we see that $a_i F Y_{i-1} = a'_i F
    Y_{i-1}$.  We have shown
    \begin{equation*}
      \{a'_i F Y_{i-1} : (a'_1,\ldots,a'_n) \in S\} = \{a_i F Y_{i-1}\}.
    \end{equation*}
    It follows that $a_i F Y_{i-1}$ is definable over $c, b_{i+1},
    \ldots, b_n$.  The fibers of the map $G/Y_{i-1} \to G/(FY_{i-1})$
    are finite, and so $a_i Y_{i-1} = b_i$ is algebraic over $c,
    b_{i+1}, \ldots, b_n$.
  \end{claimproof}
  By Claim 2 and induction, $\bar{b} \in \acl(c)$.  Therefore \[n \le
  \dpr(\bar{b}/\emptyset) \le \dpr(c/\emptyset) \le \dpr(G).\] As $n$
  was arbitrary, $G$ has infinite dp-rank, a contradiction.
\end{proof}

\subsection{The saturated case} \label{ssec-sat}

Until Subsection~\ref{ssec-unsat}, 
we will work in a monster model $\M \models
p\mathrm{CF}$.  Fix a definable group $G$, not definably compact, and fix a good
neighborhood basis $\{O_t : t \in \Gamma_\M\}$ in the sense of Definition~\ref{def-gnb}.
\begin{Lemma}\label{new-config-1}
  There is no bad gap configuration in $G$.
\end{Lemma}
\begin{proof}
  
    For definable sets in $p$CF, dp-rank agrees
    with dimension.  In particular, dp-rank is finite.  Therefore Lemma~\ref{new-config-0} applies to $G$.
\end{proof}
Recall from Definition~\ref{def-bounded} and Proposition~\ref{prop-bounded}(\ref{pb-1}) that a subset
$S \subseteq G(\M)$ is \emph{bounded} if $S \subseteq O_t$ for some $t
\in \Gamma_\M$.
\begin{Lemma}\label{ext-up}
  If $S \subseteq G(\M)$ is bounded, then $S \subseteq X$ for some
  bounded externally definable subgroup $X \subseteq G(\M)$.
\end{Lemma}
\begin{proof}
  Take $t_0 \in \Gamma_\M$ such that $S \subseteq O_{t_0}$.  By
  Proposition~\ref{up-down}, we can build an ascending sequence
  \begin{equation*}
    t_0 < t_1 < t_2 < \cdots
  \end{equation*}
  in $\Gamma_\M$ such that $O_{t_i} \cdot O_{t_i} \subseteq
  O_{t_{i+1}}$ for each $i$.  By saturation, we can also find some
  $t_\omega > t_i$ for all finite $i$.  Set $X = \bigcup_{i < \omega} O_{t_i}$.
  The set $X$ is externally definable (Remark~\ref{ext-chain}).  The
  set $X$ is bounded, because $X \subseteq O_{t_\omega}$.  We have $X
  = X^{-1}$ because $O_{t_i} = O_{t_i}^{-1}$ for each $i$.  Lastly,
  $X$ is closed under the group operation by choice of the $t_i$'s.
\end{proof}

\begin{Lemma}\label{annoying}
  Let $I$ be an unbounded subset of $G$.  Let $X \subseteq G(\M)$ be a
  bounded subgroup.  Then there is an externally definable bounded
  subgroup $X' \supseteq X$ such that $(X' \cap I)/X$ is infinite.
\end{Lemma}
\begin{proof}
  We claim that $I/X$ is infinite.  Otherwise, $I$ is contained in a
  finite union of cosets: $I \subseteq \bigcup_{i = 1}^n a_i X$.  Take
  $t \in \Gamma_M$ such that $X \subseteq O_t$.  Then $I$ is a subset
  of the definably compact set $\bigcup_{i = 1}^n a_i O_t$, so $I$ is
  bounded, a contradiction.

  Now take $a_1, a_2, a_3, \ldots \in I$ such that the cosets $a_iX$
  are pairwise distinct.  By saturation, there is some $t \in \Gamma$
  such that $\{a_1,a_2,\ldots\} \subseteq O_t$.  Then
  $\{a_1,a_2,\ldots\}$ and $X$ are bounded.  By Lemma~\ref{ext-up},
  there is an externally definable bounded subgroup $X'$ containing
  $\{a_1,a_2,\ldots\} \cup X$.  Then $(X' \cap I)/X$ is infinite,
  witnessed by the $a_iX$.
\end{proof}

Recall from Definition~\ref{d-na} that $G$ is \emph{nearly abelian} if there is a definably compact
  definable normal subgroup $K \subseteq G$ with $G/K$ abelian.
Equivalently, $G$ is nearly abelian if there is a definably compact
subgroup $K$ containing the derived group $[G,G]$.

\begin{Lemma}\label{unbounded-I}
  Suppose that $G$ is nearly abelian.  Let $I$ be an unbounded
  definable subset of $G(\M)$.  For any bounded set $A$, there is $t
  \in \Gamma_\M$ such that $I \diamond O_t \setminus A$ is bounded.
\end{Lemma}
\begin{proof}
  Suppose not.
  \begin{Claim}
    For any bounded sets $C \supseteq B \supseteq A$, the set $I
    \diamond C \setminus B$ is unbounded.
  \end{Claim}
  \begin{claimproof}
    Take $t \in \Gamma_\M$ such that $C \subseteq O_t$.  Then $I
    \diamond C \setminus B$ contains the unbounded set $I \diamond O_t
    \setminus A$, because $O_t \setminus A \supseteq C \setminus B$.
  \end{claimproof}

  Let $K$ be the normal subgroup witnessing near-abelianity.  By
  Lemma~\ref{ext-up}, there is a bounded externally definable subgroup
  $X_0 \supseteq A \cup K$.  By Lemma~\ref{annoying} we can
  recursively build an increasing chain of bounded externally
  definable subgroups
  \begin{equation*}
    X_0 \subseteq X_1 \subseteq X_2 \subseteq \cdots
  \end{equation*}
  such that
  \begin{itemize}
  \item $(X_1 \cap I)/X_0$ is infinite.
  \item For $n > 1$, $(X_n \cap (I \diamond X_{n-1} \setminus
    X_0))/X_{n-1}$ is infinite.  This is possible because $I \diamond
    X_{n-1} \setminus X_0$ is unbounded by the claim. 
  \end{itemize}
  Let $Y_i = X_0$ for all $i$, and let $F = \{\id_G\}$.  Note $X_0$ is
  normal, because it contains $K$ which contains $[G,G]$.  
  We have constructed a bad gap configuration in $G$, contradicting Lemma~\ref{new-config-1}.
\end{proof}

\begin{Lemma}\label{ext-down}
  If $S \subseteq G(\M)$ is a neighborhood of $\id_G$, then $S
  \supseteq X$ for some externally definable open subgroup $X
  \subseteq G(\M)$.  If, in addition, $B \subseteq G(\M)$ is a bounded
  set, then we can choose the group $X$ to ensure $X^B \subseteq X$.
\end{Lemma}
\begin{proof}
  Take $t_0 \in \Gamma_M$ such that $S \supseteq O_{t_0}$.  By Proposition~\ref{up-down} and Lemma~\ref{conjugation}, there is a descending
  sequence
  \begin{equation*}
    t_0 > t_1 > t_2 > \cdots
  \end{equation*}
  in $\Gamma_\M$ such that $O_{t_{i+1}} \cdot O_{t_{i+1}} \subseteq
  O_{t_i}$ and also $O_{t_{i+1}}^B \subseteq O_{t_i}$.  Take $X =
  \bigcap_{i = 1}^\infty O_{t_i}$.  Then $X$ is an externally
  definable subgroup with $X^B \subseteq X$.  We can take some
  $t_\omega$ less than all the $t_i$'s, and then $O_{t_\omega}
  \subseteq X$.  Therefore $X$ has interior, and is an open subgroup.
\end{proof}

\begin{Lemma}\label{confusion}
  Let $I$ be an unbounded definable subset of $G(\M)$.  Let $F$ be a
  finite subgroup of $G(\M)$.  Then there exist $t, t' \in \Gamma_\M$
  such that $I \diamond O_t \setminus (F \cdot O_{t'})$ is bounded.
\end{Lemma}
\begin{proof}
  Suppose not.
  \begin{Claim}
    For any neighborhood $A \ni id_G$ and any bounded set $B \subseteq
    G$, the set $I \diamond B \setminus FA$ is unbounded.
  \end{Claim}
  \begin{claimproof}
    Take $t, t'$ such that
  \begin{gather*}
    O_{t'} \subseteq A \text{ and } B \subseteq O_t \\
    O_t \setminus (F \cdot O_{t'}) \supseteq B \setminus (F \cdot A) \\
    I \diamond O_t \setminus (F \cdot O_{t'}) \subseteq I \diamond B \setminus (F \cdot A). \qedhere
  \end{gather*}
  \end{claimproof}

  Take any bounded open externally definable subgroup $X_0 \subseteq
  G$.  By Lemma~\ref{ext-down} there is an externally definable open
  subgroup $Y_0 \subseteq X_0$ such that $Y_0^F \subseteq Y_0$.
  Recursively build chains
  \begin{gather*}
    X_0 \subseteq X_1 \subseteq \cdots \\
    Y_0 \supseteq Y_1 \supseteq \cdots
  \end{gather*}
  where
  \begin{itemize}
  \item $X_i$ is a bounded externally definable subgroup, chosen large
    enough to ensure that $(X_i \cap (I \diamond X_{i-1} \setminus F Y_{i-1}))/X_{i-1}$ is infinite (Lemma~\ref{annoying}).
  \item $Y_i$ is an open externally definable subgroup with $Y_i^F =
    Y_i$, chosen small enough that $Y_i^{X_i} \subseteq Y_{i-1}$
    (Lemma~\ref{ext-down}).
  \end{itemize}
  This gives a bad gap configuration in $G$, contradicting Lemma~\ref{new-config-1}.
\end{proof}

\subsection{The general case} \label{ssec-unsat}

\begin{Prop} \label{gaps-1} 
  Let $M$ be any model of $p$CF.  Let $G$ be a definable non-compact
  group and $\{O_t : t \in \Gamma_M\}$ be a good neighborhood basis.
  Suppose that $G$ is nearly abelian.  Let $I$ be an unbounded
  definable set.  Then for any $t \in \Gamma_M$, there is $t' \in
  \Gamma_M$ such that $I \diamond O_{t'} \setminus O_t$ is bounded.
\end{Prop}

\begin{proof}
  We may replace $M$ with a monster model, and then apply
  Lemma~\ref{unbounded-I}.
\end{proof}

\begin{Prop} \label{gaps-2} 
  Let $M$ be any model of $p$CF.  Let $G$ be a definable non-compact
  group.  Let $I$ be an unbounded definable set.  Let $F$ be a finite
  subgroup of $G$.  Then for any sufficiently small $s$ and
  sufficiently large $t$, the set $I \diamond O_t \setminus (F O_s)$ is
  bounded.
\end{Prop}

\begin{proof}
  We may replace $M$ with a monster model, and then apply Lemma~\ref{confusion}.
\end{proof}

\section{Stabilizers and $\mu$-stabilizers} \label{stab-review}
In this section we review some notation and facts from
\cite{Y.-Peterzil-and-S.-Starchenko}.
\subsection{Stabilizers}
Let $G$ be a group definable in a structure $M$.

\begin{Notation}
\begin{enumerate}
  \item [(1)] If $\phi(x)$ and $\psi(x)$ are $G$-formulas then $\phi\cdot \psi$ denotes the $G$-formula
\[
(\phi\cdot \psi)(x):= \exists u\exists v(\phi(u)\wedge \psi(v)\wedge x=u\cdot v).
\]
Thus $(\phi \cdot \psi)(M) = \phi(M) \cdot \psi(M)$.
  \item [(2)] More generally, if $q(x)$ and $r(x)$ are partial $G$-types then $q\cdot r$ denotes the $G$-type
\[
(q\cdot r)(x):= \{\phi\cdot \psi(x) \mid q(x)\vdash \phi(x),~ r(x)\vdash \psi(x)\}.
\]
Thus $(q \cdot r)(N) = q(N) \cdot r(N)$ for an $|M|^+$-saturated
elementary extension $N \succ M$.
\item [(3)] If $g \in G(M)$ and $\phi(x)$ is a $G$-formula, then $g \cdot \phi$ denotes the $G$-formula
  \[
  (g \cdot \phi)(x) := \exists u (\phi(u) \wedge x = g \cdot u).
  \]
  Thus $(g \cdot \phi)(M) = g \cdot \phi(M)$.
\item [(4)] If $g \in G(M)$ and $p(x)$ is a partial $G$-type then $g \cdot p$ denotes the $G$-type
  \[
  (g \cdot p)(x) := \{g \cdot \phi(x) \mid p(x) \vdash \phi(x)\}.
  \]
  Thus $(g \cdot p)(N) = g \cdot p(N)$ for an $|M|^+$-saturated $N
  \succ M$.
\end{enumerate}
\end{Notation}
Note that for partial $G$-types $q_1, q_2, q_3$ over $M$, we have
\begin{equation*}
  (q_1 \cdot q_2) \cdot q_3 = q_1 \cdot (q_2 \cdot q_3),
\end{equation*}
as $((q_1 \cdot q_2) \cdot q_3)(N) = q_1(N) \cdot q_2(N) \cdot q_3(N)
= (q_1 \cdot (q_2 \cdot q_3))(N)$ for $|M|^+$-saturated $N \succ M$. 

\begin{Def}
Given a partial type $\Sigma(x)$ over $M$, define $\stab(\Sigma)$ to
be the stabilizer, i.e.,
\[
\stab(\Sigma) := \{g \in G(M) \mid g \Sigma \equiv \Sigma\},
\]
where $\Sigma \equiv \Sigma'$ if $\Sigma(x) \vdash \Sigma'(x)$ and
$\Sigma'(x) \vdash \Sigma(x)$.  Equivalently, $\stab(\Sigma)$ is $\{g
\in G(M) \mid g \Sigma(N) = \Sigma(N)\}$ for $|M|^+$-saturated $N 
\succeq M$.
\end{Def}

\begin{Def}\label{stab-phi}
Given a partial type $\Sigma(x)$ over $M$ and an $\la$-formula $\phi(x,y)$, we define
\[
\stab_\phi(\Sigma)=\bigcap_{b\in M^k}X_{\phi,b},
 \]
 where each $X_{\phi,b}$ is the stabilizer of $\{g\in G(M) \mid \Sigma\vdash (g\phi)(x,b)\}$.
\end{Def}
\begin{Rmk}
  Given $\phi(x;y)$, let $\phi'(x;y,z)$ be the formula $\phi(z \cdot
  x; y)$.  Then $G$ acts on $\phi'$-types by left translation, and
  $\stab_\phi(\Sigma)$ is the stabilizer of the $\phi'$-type generated
  by $\Sigma$.
\end{Rmk}
\begin{Rmk}
  Note that our $\stab_\phi$ is slightly different from the $\Stab_\phi$ considered in \cite{Y.-Peterzil-and-S.-Starchenko}, which is more like the set $X_{\phi,b}$ appearing in Definition~\ref{stab-phi} above.
\end{Rmk}

The following two facts are easy exercises. 
\begin{Fact}
$\stab_\phi(\Sigma)$ is a definable subgroup of $G$ if $\Sigma$ is definable.
\end{Fact}

\begin{Fact}\label{fact-4.7}
For every partial type $\Sigma$ over $M$.
\[
\stab(\Sigma)=\bigcap_{ \phi\in \la}\stab_\phi(\Sigma)
\]
In particular, if $\Sigma$ is definable then $\stab(\Sigma)$ is an
intersection of definable subgroups.
\end{Fact}
Recall the notation $\Sigma^N$ for the canonical extension of a
definable type $\Sigma$ to an elementary extension $N \succ M$, and
the notation $(d_\Sigma x)\phi(x;y)$ for the $\phi$-definition of
$\Sigma$.
\begin{Lemma}\label{new-lemma}
  If $\Sigma$ is definable and $N \succ M$, then $\stab_\phi(\Sigma^N)
  = \stab_\phi(\Sigma)(N)$, and so
  \begin{equation*}
    \stab(\Sigma^N) = \bigcap_{\phi \in \la} \stab_\phi(\Sigma)(N).
  \end{equation*}
\end{Lemma}
\begin{proof}
  Indeed, $\stab_\phi(\Sigma)(M)$ is defined by the formula
  \begin{equation*}
    \forall y \forall g : ((d_\Sigma z) \phi(g \cdot z; y))
    \leftrightarrow ((d_\Sigma z) \phi(x \cdot g \cdot z; y))
  \end{equation*}
  and $\stab_\phi(\Sigma^N)$ is defined by the same formula, because
  $\Sigma^N$ and $\Sigma$ have the same definition schema.
\end{proof}

\subsection{$\mu$-types and $\mu$-stabilizers}

In this section we assume that $G$ is a Hausdorff topological group
definable in $M$ with a uniformly definable basis $\{O_t \mid t\in
T\}$ of open neighborhoods of the identity. For each $N\succ M$, the
group $G(N)$ is again a topological group and the definable family
$\{O_t(N) \mid t\in T(N)\}$ again forms a basis for the open
neighborhoods of $\id_G$.

\begin{Def}
The \emph{infinitesimal type} of $G$, denoted $\mu(x)$, is the partial
type consisting of all formulas $x \in U$ with $U$ an $M$-definable neighborhood of $\id_G$.
\end{Def}
Thus, if $N \succeq M$, then $\mu(N)$ is the set of elements of $G(N)$
which are infinitesimally close to $\id_G$:
\begin{align*}
  \mu(N) &= \bigcap \{U(N) \mid U \text{ is an $M$-definable neighborhood of } \id_G\} \\
  &= \bigcap_{t \in T(M)} O_t(N).
\end{align*}

\begin{Fact}[{\cite[Corollary~2.5 and Claim~2.15]{Y.-Peterzil-and-S.-Starchenko}}]\label{some-fact}~
\begin{enumerate}
\item If $N \succ M$, then $\mu(N)$ is a subgroup of $G(N)$ normalized
  by $G(M)$.
  \item For any definable $q \in S_G(M)$, the partial type $\mu \cdot q$ is definable.
\end{enumerate}
\end{Fact}
Partial types of the form $\mu \cdot q$ for $q \in S_G(M)$ are called
\emph{$\mu$-types}.  The \emph{$\mu$-stabilizer} of $q \in S_G(M)$ is
the stabilizer of the associated $\mu$-type:
\begin{equation*}
  \stab^\mu(q) := \stab(\mu \cdot q).
\end{equation*}
Note that if $\mu$ is the infinitesimal type of $G = G(M)$, and $N \succeq M$, then the canonical extension $\mu^N$ is the infinitesimal type of $G(N)$. 

\begin{Fact}[{\cite[Remark~2.16]{Y.-Peterzil-and-S.-Starchenko}}] \label{ext-cdot}
  If $p$ is a definable type over $M$ and $N \succ M$, then the
  product of the canonical extensions is equal to the canonical
  extension of the product:
  \begin{equation*}
    \mu^N \cdot p^N = (\mu \cdot p)^N.
  \end{equation*}
\end{Fact}

\begin{Rmk} \label{useful-rmk}
  Let $N$ be an $|M|^+$-saturated extension of $M$, and $\mu^N$ and
  $p^N$ be the canonical extensions of $\mu$ and $p$.  Then
  \begin{equation*}
    \stab(\mupn) = \stab((\mup)^N) = \bigcap_{\phi \in \la} \stab_\phi(\mup)(N),
  \end{equation*}
  by Lemma~\ref{new-lemma} and Fact~\ref{ext-cdot}.
\end{Rmk}

By Fact~\ref{some-fact}, $\mu(N)\cdot G(M)$ is a subgroup of $G(N)$ as
$\mu(N)\sq G(N)$ is normalized by $G(M)$.  This subgroup is the ${\cal
  O}_{G(M)}(N)$ of Definition~\ref{o-st-def}. Because $\mu(N)\cap
G(M)=\{\id_G\}$, the group $\mu(N) \cdot G(M)$ is a semidirect
product of $\mu(N)$ and $G(M)$, and there is a natural homomorphism
\begin{equation*}
  \mathcal{O}_{G(M)}(N) = \mu(N) \cdot G(M) \to G(M).
\end{equation*}
This map is exactly the ``standard part'' map $\st_M^N$ of
Definition~\ref{o-st-def}.  For $Y\sq G(N)$, we will write
$\st_M^N(Y)$ as a shorthand for $\st_M^N(Y \cap
\mathcal{O}_{G(M)}(N))$, following \cite{Y.-Peterzil-and-S.-Starchenko}.

\begin{Lemma} \label{useful-lemma}
  Let $p\in S_G(M)$ be a definable type and let $\beta \in G(\M)$
  realize $p^N$.  Then
\begin{enumerate}
    \item \label{clause-i} $\stab(\mupn)=\stn(p^N(\M)\beta^{-1})$;
    \item \label{clause-ii} $\stab(\mupn)=\bigcap_{\psi\in p^N}\stn(\psi(\M)\beta^{-1})$;
\end{enumerate}
\end{Lemma}
\begin{proof}
Clause~(\ref{clause-i}) is by Claim 2.22 in \cite{Y.-Peterzil-and-S.-Starchenko}.

For (\ref{clause-ii}), we must show
\begin{equation*}
  \stn(p^N(\M)\beta^{-1}) = \bigcap_{\psi\in p^N}\stn(\psi(\M)\beta^{-1}).
\end{equation*}
The $\subseteq$ direction is clear.  For $\supseteq$, suppose that $g
\in \bigcap_{\psi\in p^N}\stn(\psi(\M)\beta^{-1})$.  Then for any
$\psi \in p^N$, there is $h_\psi \in \psi(\M)$ such that $h_\psi \cdot
    \beta^{-1} \cdot g^{-1}$ satisfies $\mu^N$.  By compactness
    there is $h \in p^N(\M)$ such that $h \cdot \beta^{-1} \cdot
    g^{-1}$ satisfies $\mu^N$.  Then $g \in
    \stn(p^N(\M)\beta^{-1})$.
\end{proof}

\section{Proof of main theorems} \label{sec:main}

From now on $M$ is a $p$-adically closed field, $\M\succ M$ is the
monster model, $G\sq M^n$ denotes a group definable in $M$, and $\mu$
denotes the infinitesimal type of $G$ over $M$. All formulas and types
will be $G$-formulas and $G$-types.  We assume $G$ is not definably
compact.  Fix a good neighborhood basis $\{O_t : t \in \Gamma_M\}$ of
$G$.

Fix a 1-dimensional definable type $p \in S_G(M)$ which does not
specialize to any point of $G(M)$.  Such a type $p$ exists by
Proposition~\ref{sp-char}.  Fix a small $|M|^+$-saturated model $N$
with $M \prec N \prec \M$.  As usual, $p^N$ and $\mu^N$ denote the
canonical extensions to $N$.  Fix an element $\beta \in G(\M)$
realizing $p^N$.

\begin{Rmk} \label{p-unb}
  The types $p$ and $p^N$ are ``unbounded'' in the following sense:
  \begin{enumerate}
  \item If $t \in \Gamma_M$, then $O_t \notin p$.
  \item If $t \in \Gamma_N$, then $O_t \notin p^N$.
  \item If $X$ is a bounded $M$-definable subset of $G(M)$, then $X
    \notin p$.
  \item If $X$ is a bounded $N$-definable subset of $G(N)$, then $X
    \notin p^N$.
  \end{enumerate}
  Point (1) follows by Proposition~\ref{sp-char}: if $O_t \in p$ then
  $p$ specializes to a point in $O_t(M)$, because $O_t$ is definably
  compact.  Point (2) then follows because $p^N$ is the heir of $p$. 
  Points (3) and (4) reduce to (1) and (2), respectively.
\end{Rmk}



\begin{Lemma}\label{2-equi}
  $\stab(\mupn) = \bigcap_{\phi \in p} \stn(\phi(N)\beta^{-1})$.
\end{Lemma}
\begin{proof}
  By Lemma~\ref{useful-lemma}, it suffices to show
  \begin{equation*}
    \bigcap_{\phi \in p}\stn(\phi(\M)\beta^{-1}) \subseteq
    \bigcap_{\phi \in p^N}\stn(\phi(\M)\beta^{-1})
  \end{equation*}
  Suppose $g$ belongs to the left-hand side.  In particular, $g \in
  G(N)$.  By a compactness argument similar to
  Lemma~\ref{useful-lemma}, we see that $g = \epsilon b \beta^{-1}$
  for some $\epsilon \in \mu^N(\M)$ and $b \in p(\M)$.  It suffices to
  show $b \in p^N(\M)$.  By Lemma~\ref{crux}, it suffices to show $b
  \notin O_t(\M)$ for any $t \in \Gamma_N$.  Suppose $b \in O_t(\M)$.
  Since $g \in N$, there is some $t' \in \Gamma_N$ such
  that $g^{-1} \cdot O_0(\M) \cdot O_t(\M) \subseteq O_{t'}(\M)$.
  Then
  \begin{equation*}
    \beta = g^{-1} \epsilon b \in g^{-1} O_0(\M) O_t(\M) \subseteq O_{t'}(\M),
  \end{equation*}
  contradicting the fact that $\tp(\beta/N)$ is unbounded.
\end{proof}

Note that a similar argument to the proof of Lemma 2.31 of
\cite{Yao-standard-part-map} shows the following:

\begin{Fact}\label{Fact-standard-part-dim}
Suppose that $b\in \M^k$ and $\tp(b/N)$ is definable. If $Y\sq G(\M)$ is definable over $b$ then $\stn(Y)\sq G(N)$ is definable and
\[
\dim(\stn(Y))\leq \dim(Y).
\]
\end{Fact}

\begin{Lemma}\label{Lem-standard-part-dimension}
There is an $\la$-formula $\phi$ such that $\dim(\stab_\phi(\mup))\leq 1$.
\end{Lemma}

\begin{proof}
  Take $\psi \in p$ such that $\dim(\psi(\M)) = 1$.  Then
  $\dim(\stn(\psi(\M)\beta^{-1})) \le 1$ by
  Fact~\ref{Fact-standard-part-dim}.  By Remark~\ref{useful-rmk} and
  Lemma~\ref{useful-lemma},
  \begin{equation*}
    \bigcap_{\phi \in \la} \stab_\phi(\mup)(N) = \stab(\mupn)
    \subseteq \stn(\psi(\M)\beta^{-1}).
  \end{equation*}
  The intersection on the left is directed, and the set on the right
  is definable, so by $|M|^+$-saturation of $N$ there is some $\phi \in \la$
  such that
  \begin{equation*}
    \stab_\phi(\mup)(N) \subseteq \stn(\psi(\M)\beta^{-1}).
  \end{equation*}
  Then $\dim(\stab_\phi(\mup)) \le \dim(\stn(\psi(\M)\beta^{-1})) \le
  1$.
\end{proof}

To finish our main result, we now show that each $\stab_\phi(\mup)$ is not definably compact.
\begin{Lemma}\label{unbounded-II}
  Assume $G$ is nearly abelian (Definition~\ref{d-na}).  For any
  $N$-definable set $I$ containing $\beta$, the set
  $\stn(I(\M)\beta^{-1})$ is unbounded.
\end{Lemma}
\begin{proof}
  Suppose $\stn(I(\M)\beta^{-1})$ is bounded.  Then
  $\stn(I(\M)\beta^{-1}) \subseteq O_t(N)$ for some $t \in \Gamma_N$.
  By Remark~\ref{p-unb}, $I$ is unbounded.  By Proposition~\ref{gaps-1},
  there is some $t' \in \Gamma_N$ such that the set $I^{-1} \diamond
  O_{t'} \setminus O_t$ is bounded.  Then
  \[ \beta^{-1} \notin (I^{-1} \diamond O_{t'} \setminus O_t)(\M)\]
  by Remark~\ref{p-unb}.  This means that $\beta^{-1} (O_{t'}(\M)
  \setminus O_t(\M)) \cap I(\M)^{-1} \ne \emptyset$.  Therefore there
  is $a \in O_{t'}(\M) \setminus O_t(\M)$ such that $a \beta \in
  I(\M)$.  By definable Skolem functions, we can take $a \in \dcl(N
  \beta)$.  Note $a \in I(\M)\beta^{-1}$.  By
  Lemma~\ref{specializing-types}, $\stn(a)$ exists.  Because $O_{t'}
  \setminus O_t$ is closed, we see that $\stn(a) \in O_{t'}(N)
  \setminus O_t(N)$.  This contradicts the fact that
  \begin{equation*}
    \stn(a) \in \stn(I(\M)\beta^{-1}) \subseteq O_t(N). \qedhere
  \end{equation*}
\end{proof}

\begin{Lemma} \label{n-stabilizer}
  If $G$ is nearly abelian, then the type-definable group
  $\stab(\mupn) \subseteq G(N)$ is 1-dimensional and unbounded.
\end{Lemma}

\begin{proof}
  The dimension of $\stab(\mupn)$ is at most one by
  Lemma~\ref{Lem-standard-part-dimension} and
  Remark~\ref{useful-rmk}.  If $\stab(\mupn)$ is bounded, then
  $\stab(\mupn) \subseteq O_t(N)$ for some $t \in \Gamma_N$.  By
  Lemma~\ref{2-equi}, we have
  \begin{equation*}
    \bigcap_{\psi \in p} \stn(\psi(N)\beta^{-1}) = \stab(\mupn) \subseteq O_t(N).
  \end{equation*}
  The intersection on the left is a filtered intersection of definable
  sets.  There are at most $|M|$ sets in the intersection, and $N$ is $|M|^+$-saturated.  
  Therefore there is some $\psi \in p$ such that
  $\stn(\psi(N)\beta^{-1}) \subseteq O_t(N)$, contradicting
  Lemma~\ref{unbounded-II}.  Therefore $\stab(\mupn)$ is unbounded.
  In particular, it is infinite, so it has dimension at least 1.
\end{proof}

\begin{Lemma}\label{near-abelian-almost-there}
  Suppose $G$ is nearly abelian.  Then there is $\phi\in \la$ such
  that the $M$-definable group $\stab_\phi(\mup)$ is not definably
  compact and has dimension $1$.
\end{Lemma}

\begin{proof}
  By Lemma~\ref{Lem-standard-part-dimension} there is an
  $\la$-formula $\phi$ such that $\dim(\stab_\phi(\mup)) \le 1$.  By
  Remark~\ref{useful-rmk}, we have
  \begin{equation*}
    \stab_\phi(\mup)(N) \supseteq \stab(\mupn),
  \end{equation*}
  and therefore $\stab_\phi(\mup)$ is unbounded by
  Lemma~\ref{n-stabilizer}.  In particular, $\stab_\phi(\mup)$ is
  infinite, and $\dim(\stab_\phi(\mup)) \ge 1$.
\end{proof}

\begin{Thm}\label{near-abelian-main-theorem}
  Let $G$ be a definable group in a $p$-adically closed field $M$.  Suppose $G$ is nearly abelian, and not definably compact.
  \begin{enumerate}
      \item \label{namt-1} $G$ has a one-dimensional definable subgroup which is not definably compact.
      \item \label{namt-2} If $p \in S_G(M)$ is a definable unbounded 1-dimensional type, then there is $\phi \in \la$ such that $\stab_\phi(\mup)$ is a one-dimensional definable subgroup of $G$ which is not definably compact.
      \item \label{namt-3} Suppose in addition that $M$ is $\aleph_1$-saturated.  If $p \in S_G(M)$ is a definable unbounded 1-dimensional type, then $\stab^\mu(p)$ is a one-dimensional type-definable subgroup of $G$ which is unbounded.
  \end{enumerate}
\end{Thm}

\begin{proof}
  Part~(\ref{namt-2}) is Lemma~\ref{near-abelian-almost-there}.  Part~(\ref{namt-1}) then follows because there is at least one unbounded 1-dimensional definable type by Proposition~\ref{sp-char}. For Part~(\ref{namt-3}), take a countable model $M_0 \preceq M$ such that $G$ and $p$ are $M_0$-definable.  Then apply Lemma~\ref{n-stabilizer} to $M$ and $M_0$ in place of $N$ and $M$ (respectively), to see that $\stab(\mup)$ is 1-dimensional and unbounded.  Type-definability is by Fact~\ref{fact-4.7}.
\end{proof}

Recall from \cite{Definable-f-Generic-Groups-over-p-Adic-Numbers} that in an NIP context, a global type $p\in S_G(\M)$ is said to be a
definable $f$-generic, abbreviated as \emph{dfg},  if there is a small submodel $M_0$ such that every left $G$-translate of $p$ is definable over $M_0$.  In  \cite{Definable-f-Generic-Groups-over-p-Adic-Numbers}, Pillay and the second author showed that: 
\begin{Fact}
A group $G$ definable over $\Q$ has \emph{dfg} iff there is a normal sequence of definable subgroups
\[
G_0\lhd ... G_i\lhd G_{i+1} ...\lhd G_n
\]
such that $G_0$ is finite, $G_n$ is a finite index subgroup of $G$, and each $G_{i+1}/G_{i}$
is definably isomorphic to either the additive group $\mathbb G_a$, or a finite index subgroup of the multiplicative
group $\mathbb G_m$.
\end{Fact}
The intuition is that ``\emph{dfg}'' means ``totally non-compact''  in the $p$-adic context.  

\begin{Lemma}
  Let $G$ be a one-dimensional definable group.  If $G$ is not definably compact, then $G$ has \textit{dfg}.
\end{Lemma}
\begin{proof}
  Recall from \cite[Section 4]{HPP} that $G$ has \emph{finitely satisfiable generics} (\emph{fsg}) if there is a small model $M_0$ and a global type $p(x)$ in $G$ such that every left $G$-translate of $p$ is finitely satisfiable in $M_0$.  By \cite[Lemma 2.9]{Definable-f-Generic-Groups-over-p-Adic-Numbers}, $G$ has \emph{fsg} or \emph{dfg}.\footnote{In \cite{Definable-f-Generic-Groups-over-p-Adic-Numbers} this is stated only for groups definable over $\Q$.  The assumption is used in order to apply \cite[Theorem~2.4]{note-on-p-adic-groups}.  However, \cite[Remark~2.5]{note-on-p-adic-groups} shows that this assumption is unnecessary.}  It suffices to show that $G$ does not have \emph{fsg}.  Suppose otherwise, witnessed by $p$ and $M_0$.  Recall that a definable set $X \subseteq G$ is \emph{generic} if finitely many left translates cover $G$.  By \cite[Proposition~4.2]{HPP}, the complement of a non-generic set is generic, and every  generic set intersects $G(M_0)$.  Let $\{W_\gamma\}_{\gamma}$ be a $\Gamma$-exhaustion of $G$.  By taking $\gamma$ sufficiently large, we can arrange for $G(M_0) \subseteq W_\gamma$, because $M_0$ is small.  Then $G \setminus W_\gamma$ does not intersect $G(M_0)$, so it is not generic.  Therefore the complement $W_\gamma$ is generic.  A finite union of left translates of $W_\gamma$ will be definably compact, so it cannot be all of $G$.  We conclude that $G$ does not have \emph{fsg}, and instead has \emph{dfg}.
\end{proof}
The following is then a corollary of Theorem~\ref{near-abelian-main-theorem}.
\begin{Cor}\label{old-5.9}
  Let $G\sq M^k$ be a nearly abelian definable group.  If $G$ is not
  definably compact, then $G$ has a one-dimensional \textit{dfg} subgroup. 
\end{Cor}

Next, we consider the general non-abelian case.  Let $G, M, N, \M, p, \beta$ be as in the start of this section.

\begin{Lemma}\label{unbounded-II.5}
  For any $N$-definable set $I$ containing $\beta$, and any finite $N$-definable
  subgroup $F \subseteq G$, the set $\stn(I(\M)\beta^{-1})$ contains a
  point outside of $F$.
\end{Lemma}
\begin{proof}
  As in Lemma~\ref{unbounded-II}, $I$ is unbounded.  Let $J$ be the
  unbounded set $I^{-1}$.  By Proposition~\ref{gaps-2}, there are $s, t
  \in \Gamma_N$ such that $J \diamond O_t \setminus F O_s$ is bounded.
  Then $\beta^{-1}
  \notin J \diamond O_t \setminus F O_s$.  Therefore
  \begin{equation*}
    \beta^{-1} (O_t \setminus F O_s) \cap J \ne \emptyset,
  \end{equation*}
  or equivalently
  \begin{equation*}
    (O_t \setminus O_s F) \beta \cap I \ne \emptyset.
  \end{equation*}
  Therefore there is $a \in O_t(\M) \setminus O_s(\M) F(\M)$ such that
  $a \beta \in I(\M)$.  By definable Skolem functions, we can take $a
  \in \dcl(N \beta)$.  Note $a \in I(\M)\beta^{-1}$.  Because $O_{t}
  \setminus O_s F$ is definably compact,
  Lemma~\ref{specializing-types} implies that $\stn(a)$ exists and is
  in $O_t \setminus O_s F$.  Then $\stn(a)$ is not in $F$, since $F
  \subseteq O_s F$.
\end{proof}

\begin{Lemma}\label{old-5.11}
  The group $\stab(\mupn)$ is 1-dimensional.  In particular, it is
  infinite.
\end{Lemma}

\begin{proof}
  As in Lemma~\ref{n-stabilizer}, the dimension is at most 1.  If
  $\dim(\stab(\mupn)) = 0$, then $\stab(\mupn)$ is a finite subgroup
  $F$.  By Lemma~\ref{2-equi}, we have
  \begin{equation*}
    \bigcap_{\phi \in p} \stn(\phi(N)\beta^{-1}) = \stab(\mupn) = F.
  \end{equation*}
  As in Lemma~\ref{n-stabilizer}, there is some $\phi \in p$ such
  that $\stn(\phi(N)\beta^{-1}) \subseteq F$.  This contradicts
  Lemma~\ref{unbounded-II.5}
\end{proof}

By Lemma~\ref{Lem-standard-part-dimension} and
Remark~\ref{useful-rmk}, we conclude
\begin{Lemma}\label{old-5.12}
  There is $\phi \in \la$ such that the $M$-definable group
  $\stab_\phi(\mup)$ has dimension 1.
\end{Lemma}
  We summarize the non-abelian case in the following theorem.
\begin{Thm}\label{end-theorem}
  Let $G$ be a definable group in a $p$-adically closed field $M$.  Suppose $G$ is not definably compact.  Let $p \in S_G(M)$ be a definable unbounded 1-dimensional type.  %
  \begin{enumerate}
      \item \label{mt-1} There is $\phi \in \la$ such that $\stab_\phi(\mup)$ has dimension 1.
      \item \label{mt-2} If $M$ is $\aleph_1$-saturated, then $\stab^\mu(p)$ is a one-dimensional type-definable subgroup of $G$.
  \end{enumerate}
\end{Thm}

\begin{proof}
Part~(\ref{mt-1}) is Lemma~\ref{old-5.12}.  For Part~(\ref{mt-2}), take a countable submodel $M_0 \preceq M$ such that $G$ and $p$ are $M_0$-definable.  Then apply Lemma~\ref{old-5.11} with $M$ and $M_0$ in place of $N$ and $M$.
\end{proof}

\begin{Acks}
The first author was supported by the National Natural Science Foundation of China (Grant No. 12101131).  The second author was supported by the National Social Science Fund of China (Grant No. 20CZX050).
\end{Acks}

\bibliographystyle{plain} \bibliography{bibliography}{}

\end{document}